\documentclass[aos,preprint]{imsart}
\usepackage{amsmath,amsthm,amsfonts,graphicx}
\usepackage[psamsfonts]{amssymb}
\usepackage[OT1]{fontenc}
\usepackage{mathtools}
\usepackage{bbm}
\usepackage{hyperref}

\newtheorem{thm}{Theorem}[section]
\newtheorem{lem}{Lemma}[section]
\newtheorem{Def}{Definition}[section]
\newtheorem{prop}{Proposition}[section]

\renewcommand{\c}{\mathbf{c}}
\newcommand{\m}{\mathbf{m}}
\def\Var{\mathop{\rm Var}}

\begin{document}
\begin{frontmatter}
\title{Margin conditions for vector quantization}
\runtitle{Margin conditions for vector quantization}

\begin{aug}
\author{Cl\'{e}ment Levrard\ead[label=e1]{clement.levrard@math.u-psud.fr}}

\runauthor{C. Levrard}
\affiliation{Universit\'{e} Paris Sud, UPMC and INRIA}
\end{aug}

\begin{abstract}

Recent results in quantization theory show that the convergence rate for the mean-squared
expected distortion of the empirical risk minimizer strategy, for any fixed probability distribution satisfying some regularity conditions, is $\mathcal{O}(1/n)$, where $n$ is the sample size (see, e.g., \cite{Chichi13} or \cite{Levrard12}). However, the dependency of the average distortion on other parameters is not known. 

This paper offers more general conditions, which may be thought of as margin conditions (see, e.g., \cite{Tsybakov99}), under which a sharp upper bound on the expected distortion rate of the empirically optimal quantizer is derived. This upper bound is also proved to be  sharp with respect to the dependency of the distortion on other natural parameters of the quantization issue.

\end{abstract}

         \end{frontmatter}
       
         \section{Introduction}\label{Introduction}
         
         Quantization, also called lossy data compression in information theory, is the problem of replacing a probability distribution with an efficient and compact representation, that is a finite set of points. To be more precise, let $P$ denote a probability distribution over $\mathbb{R}^d$ and $k$ a positive integer. A so-called $k$-quantizer $Q$ is a map from $\mathbb{R}^d$ to $\mathbb{R}^d$, whose image set is made of exactly $k$ points, that is $\left |Q(\mathbb{R}^d) \right | = k$. For such a quantizer, every image point $c_i$ $\in$ $Q\left (\mathbb{R}^d \right )$ is called a code point, and the vector composed of the code points $(c_1, \hdots, c_k)$ is called a codebook.  By considering the preimages of its code points, a quantizer $Q$  partitions the Euclidean space $\mathbb{R}^d$ into $k$ groups, and assigns each group a representative. General references on the subject are to be found in \cite{GL00}, \cite{Gersho91} and \cite{Linder02} among others. 
				
         The quantization theory was originally developed as a way to answer signal compression issues in the late 40's (see, e.g., \cite{Gersho91}). However, unsupervised classification is also in the scope of its application. Isolating meaningful groups from a cloud of data is a topic of interest in many fields, from social science to biology. Classifying points into dissimilar groups of similar items is as more interesting as the amount of accessible data is large. In many cases data need to be preprocessed through a quantization algorithm in order to be exploited.
         
         If the distribution $P$ has a finite second moment, the performance of a quantizer $Q$ is measured by the risk, or distortion
         \[
         R(Q) := P \| x - Q(x) \| ^2 ,
         \] 
         where $Pf$ means integration of the function $f$ with respect to $P$. The choice of the Euclidean squared norm is convenient, since it takes advantages of the Euclidean space structure of $\mathbb{R}^d$. Nevertheless, it is worth pointing out that several authors deal with more general distortion functions. For further information on this topic, the interested reader is referred to \cite{GL00} or \cite{Fischer10}.
				
				In order to minimize the distortion introduced above, it is clear that only quantizers of the type $x \mapsto \arg\min_{c_1, \hdots, c_k}{\|x - c_i\|^2}$ are to be considered. Such quantizers are called nearest-neighbor quantizers. With a slight abuse of notation, $R(\c)$ will denote the risk of the nearest-neighbor quantizer associated with a codebook $\c$.

         Provided that $P\|x\|^2 < \infty $, there exist optimal codebooks minimizing the risk $R$ (see, e.g., Lemma 8 in \cite{Pollard82b} or  Theorem 4.12 in \cite{GL00}). The aim is to design a codebook $\hat{\c}_n$, according to a $n$-sample drawn from $P$, whose distortion is as close as possible to the optimal distortion $R(\c^*)$, where $\c^*$ denotes an optimal codebook. 
         
         To solve this problem, most approaches to date attempt to implement the principle of empirical risk minimization in the vector quantization context. Let $X_1, \hdots, X_n$ denote an independent and identically distributed sample with distribution $P$. According to this principle, good code points can be found by searching for ones that minimize the empirical distortion over the training data, defined by 
         \[
         \hat{R}_n(\mathbf{c}) := \frac{1}{n}\sum_{i=1}^{n}{\left \|X_i - Q(X_i)\right \|^2} = \frac{1}{n}\sum_{i=1}^{n}{ \min_{j=1, \hdots, k}{\|X_i - c_j\|^2}}.
         \]  
          If the training data represents the source well, then $\hat{\mathbf{c}}_n$ will hopefully also perform near optimally on the real source, that is $\ell(\hat{\c}_n,\c^*) = R(\hat{\c}_n)-R(\c^*) \approx 0$. The problem of quantifying how good empirically designed codebooks are, compared to the truly optimal ones, has been extensively studied, as for instance in \cite{Linder02}.  
          
          It has been proved in \cite{Linder94} that $\mathbb{E}\ell(\hat{\c}_n,\c^*) = \mathcal{O}(1/\sqrt{n})$, provided that $P$ has a finite second moment. However, this upper bound can be tightened whenever the source distribution satisfies additional assumptions. 
          
             For the special case of finitely supported distributions, it is shown in \cite{Antos04} that  $\mathbb{E}\ell(\hat{\c}_n,\c^*) = \mathcal{O}(1/n)$. There are much more results in the case where $P$ is assumed to have a density. 
             
             In fact, different sets of assumptions have been introduced in \cite{Antos04}, \cite{Pollard82} or \cite{Levrard12}, for the loss $\mathbb{E}\ell(\hat{\c}_n,\c^*)$ to decrease at the rate $\mathcal{O}(1/n)$ in the density case. As shown in \cite{Levrard12}, these different sets of assumptions turn out to be equivalent to a technical condition, similar to that used in \cite{Massart06} to derive fast rates of convergence in the statistical learning framework.
						
						Thus, a question of interest is to know whether some margin type conditions can be derived for the source distribution to satisfy the technical condition mentioned above, as has been done in the statistical learning framework in \cite{Tsybakov99}.
					
						 Theorem 3.2 of \cite{Levrard12} offers a partial answer, proving that a sufficient condition is that $P$ is divided into $k$ well separated areas. However, this condition is not fully satisfactory, since it consists in a bound on the density located at the $d-1$ dimensional region between optimal code cells, whereas margin conditions in the statistical learning framework are bounds on the weight with respect to $P$ of the $\varepsilon$-neighborhood of the critical value $1/2$ for the regression function.
             
             Next, the scope of Theorem 3.2 of \cite{Levrard12} is constrained to distributions with continuous densities, whereas margin conditions in \cite{Tsybakov99} do not require regularity of the regression function.
             
             This paper addresses both these issues, providing a condition which can  clearly be thought of as a margin condition in the quantization framework, under which the loss $\mathbb{E}\ell(\hat{\c}_n,\c^*) = \mathcal{O}(1/n)$. 
						
						Moreover, some explicit oracle inequality is derived in this case, that is an upper bound of the form $\mathbb{E}\ell(\hat{\c}_n,\c^*) \leq C(k,d,P)/n$, where the dependency of $C(k,d,P)$ on its parameters is explicit, developing the technique used in \cite{Levrard12} or \cite{Chichi13}. It is worth pointing out that the parameters mentioned in this result, such as the smaller distance between two optimal code points, are rather natural from the quantization point of view. 
             
             In addition, this result allows to partially answer the problem mentioned in \cite{Antos05} about the minimax rates over distributions satisfying Pollard's condition. This rate has been proved in \cite{Antos05} to be $1/\sqrt{n}$, which is at first sight contradictory with the individual convergence rate of $1/n$ derived for every distribution in this case.

              The paper is organized as follows. In Section \ref{Notation} some notation and definition are introduced, as well as the so-called margin conditions. The main results are exposed in Section \ref{Results}: firstly an oracle inequality on the loss is stated, along with a minimax result, then it is shown that Gaussian mixtures are in the scope of the margin conditions. Finally, proofs are gathered in Section \ref{Proofs}, and the proofs of technical intermediate results are to be found in Section \ref{Technical results}.

         \section{Notation and Definitions}\label{Notation}
				 
				Throughout the paper, for $M >0$ and $a$ in $\mathbb{R}^d$, $\mathcal{B}(a,M)$ will denote the closed Euclidean ball with center $a$ and radius $M$. With a slight abuse of notation, $P$ is said to be $M$-bounded if its support is included in $\mathcal{B}(0,M)$.
				
         To frame the quantization issue as an empirical risk minimization issue, the following contrast function $\gamma$ is introduced as
         
         \begin{align*}
          \gamma : \left \{
                 \begin{array}{@{}ccl@{}}
                 \left (\mathbb{R}^d \right )^k \times \mathbb{R}^d& \longrightarrow & \mathbb{R} \\
                 \hspace{0.45cm}  (\mathbf{c},x) & \longmapsto & \underset{j=1, \hdots, k}{\min}{\left \| x-c_j \right \|^2}
                 \end{array}
                 \right . ,
          \end{align*}
          where $\c = (c_1, \hdots, c_k)$ denotes a codebook, that is a $k d$-dimensional vector. The risk $R(\c)$ then takes the form $R(\c)= R(Q) = P\gamma(\c,.)$, where we recall that $Pf$ denotes the integration of the function $f$ with respect to $P$. Similarly, the empirical risk $\hat{R}_n(\c)$ can be defined as $\hat{R}_n(\c) = P_n\gamma(\c,.)$, where $P_n$ is the empirical distribution associated with $X_1, \hdots, X_n$, in other words $P_n(A) = 1/n \left | \{i |X_i \in A \} \right |$, for every measurable subset $A \subset \mathbb{R}^d$. 
					
					It is worth pointing out that, if $P\|x\|^2 < \infty$, then there exist such minimizers $\hat{\c}_n$ and $\mathbf{c}^*$ (see, e.g., Theorem 4.12 in \cite{GL00}). In the sequel the set of minimizers of the risk $R(.)$ will be denoted by $\mathcal{M}$.

            Let $c_1, \hdots, c_k$ be a sequence of code points. A central role is played by the set of points which are closer to $c_i$ than to any other $c_j$'s. To be more precise, the Voronoi cell, or quantization cell associated with $c_i$ is the closed set defined by
         \begin{align*}\label{Voronoidefinition}
         V_i(\c)= \left \{ x \in \mathbb{R}^d | \quad  \forall j \neq i \quad  \|x-c_i\| \leq \|x-c_j\|\right \}. 
         \end{align*}
         It may be noted that $(V_1(\c), \hdots, V_k(\c))$ does not form a partition of $\mathbb{R}^d$, since $V_i(\c) \cap V_j(\c) $ may be non empty. To address this issue, a Voronoi partition associated with $\c$ is defined as a sequence of subsets $(W_1(\c), \hdots, W_k(\c))$ which forms a partition of $\mathbb{R}^d$, and such that for every $i=1, \hdots, k$,
         \[
         \bar{W}_i(\c)=V_i(\c),
         \]
         where $\bar{W}_i(\c)$ denotes the closure of the subset $W_i(\c)$. The open Voronoi cell is defined the same way by
         \begin{align*}
         \overset{o}{V}_i(\c)=\left \{ x \in \mathbb{R}^d | \quad  \forall j \neq i \quad  \|x-c_i\| < \|x-c_j\|\right \}.
         \end{align*} 
         
          Given a Voronoi partition $W(\c)=(W_1(\c), \hdots, W_k(\c))$, the following inclusion holds, for $i$ in $\left \{1,\hdots,k\right \}$,
          \[
          \overset{o}{V}_i(\c) \subset W_i(\c) \subset V_i(\c),
          \]
          and the risk $R(\c)$ takes the form
         \[
         R(\c)= \sum_{i=1}^{k}{P\left ( \|x-c_i\|^2 \mathbbm{1}_{W_i(\c)}(x) \right )},
         \]
         where $\mathbbm{1}_A$ denotes the indicator function associated with $A$. In the case where $(W_1, \hdots, W_k)$ are fixed subsets such that $P(W_i) \neq 0$, for every $i=1, \hdots, k$, it is clear that
         \[
         P(\|x - c_i\|^2 \mathbbm{1}_{W_i(\c)}(x)) \geq P(\|x - \eta_i\|^2 \mathbbm{1}_{W_i(\c)}(x)),
         \]
         where $\eta_i$ denotes the conditional expectation of $P$ over the subset $W_i(\c)$, that is 
         \[
         \eta_i = \frac{P(x\mathbbm{1}_{W_i(\c)}(x))}{P(W_i(\c))}.
         \] 
         Moreover, it is proved in Theorem 4.1 of \cite{GL00} that, for every Voronoi partition $W(\c^*)$ associated with an optimal codebook $\c^*$, and every $i=1, \hdots, k$, $P(W_i(\c^*)) \neq 0$. Consequently, any optimal codebook satisfies the so-called centroid condition (see, e.g., Section 6.2 of \cite{Gersho91}), that is
         \[
         \c^*_i= \frac{P(x\mathbbm{1}_{W_i(\c^*)}(x))}{P(W_i(\c^*))}.
         \]
         As a remark, the centroid condition ensures that, for every $\c^*$ in $\mathcal{M}$ and $i \neq j$, 
         \begin{align*}
         P(V_i(\c^*) \cap V_j(\c^*)) = P\left ( \left \{ x \in \mathbb{R}^d| \quad \|x-c_i^*\|= \|x - c_j^*\| \right \} \right ) = 0.
         \end{align*}
         A proof of this statement can be found in Theorem 4.2 of \cite{GL00}. According to this remark, it is clear that, for every optimal Voronoi partition $(W_1(\c^*), \hdots, W_k(\c^*))$,
         \begin{align}\label{nomansland}
         \left \{
         \begin{array}{@{}ccc}
         P(W_i(\c^*)) & = &P(V_i(\c^*)) \\
         P_n(W_i(\c^*)) & \underset{a.s.}{=} & P_n(V_i(\c^*)).
         \end{array}
         \right .
         \end{align} 
				
				The following quantities are of importance in the bounds exposed in Section \ref{mainresult}:
				\begin{align}\label{Betpmin}
				\left \{
				\begin{array}{@{}ccc}
				B &=& \min_{\c^* \in \mathcal{M}, i \neq j}{\| c_i^* - c_j^*\|}\\
				p_{min} &=& \min_{\c^* \in \mathcal{M}, i =1,\hdots,k}{P(V_i(\c^*))}.
				\end{array}
				\right .
				\end{align}
         
         The role of the boundaries between optimal Voronoi cells may be compared to the role played by the critical value $1/2$ for the regression function in the statistical learning framework. To draw this comparison, the following set is introduced, for any $\c^*$ $\in$ $\mathcal{M}$,
         \begin{align*}
         N(\c^*) = \bigcup_{i \neq j}{V_i(\c^*) \cap V_j(\c^*)}.
         \end{align*} 
         Next, the critical region $N^*$ is defined as
         \begin{align*}
         N^* = \bigcup_{\c^* \in \mathcal{M}}{N(\c^*)}.
         \end{align*}
         This region seems to be of importance when considering the conditions under which the empirical risk minimization strategy for the quantization issue achieves faster rates of convergence, as exposed in \cite{Levrard12}. However, to fully draw the comparison between the margin conditions for the statistical learning issue (see, e.g., \cite{Tsybakov99}) and quantization, the neighborhood of this region has to be introduced. For this purpose the $t$-neighborhood of the critical region is defined as
         \begin{align*}
         N^*_t=  \left \{ x \in \mathbb{R}^d | \quad d (x, N^*) \leq t \right \}.
         \end{align*}
         Intuitively, if $P(N^*_t)$ is small enough, then the source distribution $P$ is concentrated around its optimal codebook, and may be thought of as a slight modification of the probability distribution with finite support made of an optimal codebook $\c^*$. To be more precise, let introduce the following key assumption: 
         \begin{Def}[Margin condition]\label{margincondition}
         Denote by $p(t) = P(N^*(t))$. Then $P$ satisfies a margin condition with radius $r_0$ if and only if
				\begin{itemize}
				\item[$i)$] $P$ is bounded by $M$,
					\item[$ii)$] $\mathcal{M}$ is finite,
					\item[$iii)$] For all $0 \leq t \leq r_0$,
					\begin{align}\label{majorationkappa}
					p(t) \leq \frac{B p_{min}}{128 M^2}t.
					\end{align}
				\end{itemize}
         \end{Def}			
				Contrary to the conditions required in \cite{Tsybakov99} in the framework of supervised classification, the margin condition introduced here only requires a local control of the weight of the neighborhood of the critical region. It is quite obvious that a global margin condition of the type $p(t) \leq B p_{min}/128 M^2 t$ for every $t >0$ implies the condition defined above. However, requiring only a local control of the weight function $p(t)$ enlarges the scope of our results, since it allows to deal with non continuous probability distributions. This point is illustrated in the following example:
				
				\textbf{Example 1}: Assume that there exists $r>0$ such that $p(x)=0$ if $x\leq r$ (for instance if $P$ is supported on $k$ points). Then $P$ satisfies a margin condition with radius $r$.

         It is also worth pointing out that the condition mentioned in \cite{Tsybakov99} requires a control of the weight of the neighborhood of the critical value $1/2$ with a polynomial function with degree larger than $1$. In the quantization framework, the special role played by the exponent $1$ leads to only consider linear controls of the weight function. This point is explained by the following example:
				
				\textbf{Example 2}: Assume that $P$ is bounded by $M$, and that there exists $Q>0$ and $q>1$ such that $p(x) \leq Q x^q$. Then $P$ satisfies \eqref{majorationkappa}, with 
         \[
         r_0 = \frac{B}{4 \sqrt{2} M} \left ( \frac{p_{min} B}{16 \sqrt{2}M Q} \right )^{1/(q-1)}.
         \]
				
				In the case where $P$ has a density, the condition \eqref{majorationkappa} can be thought of as a generalization of the condition mentioned in Theorem 3.2 of \cite{Levrard12}, which requires the density of the distribution to be small enough over the critical region. In fact, provided that $P$ has a continuous density, a uniform bound on the density over the critical region provides a local control of the weight function with a polynomial function of degree 1. This idea is developed in the following example:
					
		\textbf{Example 3}(Continuous densities): Assume that $P$ has a continuous density $f$, is bounded by $M$, and that $\mathcal{M}$ is finite. Moreover, assume that
         \begin{align}\label{continuousdensity}
         \int_{N^*}{f(u)du} < \frac{B p_{min}}{128 M^2}.
         \end{align}
         Then, by considering the derivative at $0$ of the map $t \mapsto p(t)$, there exists $r_0 >0$ such that $P$ satisfies a margin condition with radius $r_0 >0$. It can easily be deduced from \eqref{continuousdensity} that an uniform bound on the density located at the critical region can provide a sufficient condition for a distribution $P$ to satisfy  a margin condition. Such a result has to be compared to Theorem 3.2 of \cite{Levrard12}, where it was required that
         \[
         \|f_{\left | N^* \right .}\|_{\infty} \leq \frac{\Gamma\left ( \frac{d}{2} \right ) B }{2^{d+5} M^{d+1} \pi^{d/2}}p_{min},
         \]
         where $\Gamma$ denotes the Gamma function.

				Another interesting parameter of the quantization issue is the following separation factor, which quantifies the difference between optimal codebooks and local minimizers of the risk.
				\begin{Def}\label{epsilonseparation}
          Denote by $\tilde{\mathcal{M}}$ the set of local minimizers of the map $\c \longmapsto P\gamma(\c,.)$. Then $P$ is said to be $\varepsilon$-separated if 
          \begin{align}\label{separationfactor}
          \inf_{\c \in \tilde{\mathcal{M}} \cap \mathcal{M}^c}{\ell(\c,\c^*)} = \varepsilon.
          \end{align}
					\end{Def}
          It may be noticed that local minimizers of the risk function satisfy the centroid condition. Whenever $P$ has a density and $P\|x\|^2 < \infty$, it can be proved that the set of minimizers of $R$ coincides with the set of codebooks satisfying the centroid condition, also called stationary points (see, e.g., Lemma A of \cite{Pollard82}). However, this result cannot be extended to non continuous distributions, as proved in Example 4.11 of \cite{GL00}. 
					
			The main results of the present paper are based on the following proposition, which connects the margin condition stated in Definition \ref{margincondition} to the condition introduced in Theorem 2 of \cite{Antos04}.
             
             \begin{prop}\label{lienconditionmarge}
            
         Assume that $P$ satisfies a margin condition with radius $r_0$, and is $\varepsilon$-separated. Then, for every codebook $\c$ in $\mathcal{B}(0,M)^k$,
         \[
         \|\c - \c^*(\c)\|^2 \leq \kappa_0 \ell(\c,\c^*),
         \]
         where $\c^*(\c) \in \arg\min_{\c^* \in \mathcal{M}}{\|\c - \c^*\|}$,  and  $\kappa_0 = 4kM^2\left ( \frac{1}{\varepsilon} \vee \frac{64 M^2}{p_{min} B^2 r_0^2} \right )$.
         \end{prop}	         
					As mentioned in \cite{Chichi13} or \cite{Levrard12}, the connection between the loss and the Euclidean squared distance can be thought of as a technical margin condition. It is worth pointing out that the dependency of $\kappa_0$ on different parameters of the quantization issue is explicit. This point allows us to derive explicit upper bounds on the excess risk in the following section.

         \section{Results}\label{Results}
				
				\subsection{Risk bound}\label{riskbound}

         The main result of this paper is the following:

         \begin{thm}\label{mainresult}
          Assume that $P$ satisfies a margin condition with radius $r_0$, and is $\varepsilon$-separated. Let $\kappa_0$ be defined as
					\[
					\kappa_0 = 4kM^2\left ( \frac{1}{\varepsilon} \vee \frac{64 M^2}{p_{min} B^2 r_0^2} \right ).
					\]
					If $\hat{\c}_n$ is an empirical risk minimizer, then, with probability larger than $1-2e^{-x}$,
         \begin{align}\label{firstinequality}
         \ell(\hat{\c}_n,\c^*) \leq C_0 \kappa_0 \frac{\left |\mathcal{M} \right |^2  R(\c^*)}{n} + \kappa_0^3 \frac{C_1 }{n^2} + \kappa_0 \frac{C_2}{n}x + \frac{C_3}{n}x,
         \end{align}
         where $C_0$ is an absolute constant, $C_1$ is a combination of square roots of polynomial functions in $k$, $\log(k)$, $d$, $B$ and $M$, $C_2$ is polynomial in $k$ and $M$, $C_3$ is polynomial in $M$ and $\sqrt{k}$.
         
         Moreover, under the same conditions, with probability larger than $1-2e^{-x}$,
         \begin{align}\label{secondinequality}
          \ell(\hat{\c}_n,\c^*) \leq C'_0 \kappa_0 \frac{M^2 kd \left (\log(4|\mathcal{M}| \sqrt{kd}) +1 \right )}{n} +  \kappa_0 \frac{144 M^2}{n}x + \frac{64 M^2}{n}x,
          \end{align}
          where $C_0'$ is an absolute constant.
         \end{thm}
         
         This result is in line with Theorem 3.1 in \cite{Levrard12} or Theorem 1 in \cite{Chichi13}, concerning the dependency on the sample size $n$ of the loss $\ell(\hat{\c}_n,\c^*)$. The main advance lies in the dependency on other parameters of the loss of $\hat{\c}_n$, which provides a non-asymptotic bound for the excess risk.
        
        In fact,  \eqref{secondinequality} derives from chaining arguments such as one used in \cite{Levrard12} or \cite{Chichi13}, and involves a classical dimension term of $kd$. When considering \eqref{firstinequality}, it seems that this $kd$ term disappears from the dominant term of the upper bound. This suggests that the dimension of the Euclidean space in the finite-dimensional case plays a minor role, as pointed out in Theorem 2.1 in \cite{Biau08}. An open question is to know whether such fast rates bounds can be derived in the infinite dimensional case.
				
				However, \eqref{firstinequality} may be thought of as a semi-asymptotic bound, since it involves a dominant term  and a residual term with respect to the sample size $n$. Although the dependency on other parameters of the dominant term is sharper in \eqref{firstinequality} than in \eqref{secondinequality}, the residual term $C_1/n^2$ in \eqref{firstinequality} still involves the dimension $d$. Consequently, \eqref{firstinequality} only guarantees that $\mathbb{E}\ell(\hat{\c}_n,\c^*)$ can be bounded from above with a dimension-free term when $n$ grows to infinity.
        
        Another interesting point is that Theorem \ref{mainresult} does no require $P$ to have a density, contrary to the requirements of previous results in \cite{Levrard12} or \cite{Chichi13}. This remark makes the link between bounds obtained for point-wise distributions in \cite{Antos04} and bounds for distributions with densities as in \cite{Antos04} or \cite{Levrard12}.
        
        It is also worth mentioning that the dependency in $\varepsilon$ surprisingly turns out to be sharp, as will be shown in Proposition \ref{minimax}. In fact, tuning this separation factor is the core of the demonstration of the minimax results in \cite{Bartlett98} or \cite{Antos05}. 
				
				\subsection{Minimax lower bound}  
											
					Theorem 1 in \cite{Bartlett98} ensures that the minimax convergence rate over the distributions bounded by $M$ of any empirically designed codebook can be bounded from below by $\mathcal{O}(1/\sqrt{n})$. A question of interest is to know whether this lower bound can be refined when considering only distributions satisfying some fast-convergence condition. A partial answer is given by Corollary 2 in \cite{Antos05}, where it is proved that the minimax rate over distributions with continuous densities with individual convergence rate of $\mathcal{O}(1/n)$ for the empirical risk minimizer is still $\mathcal{O}(1/\sqrt{n})$. However, since no non-asymptotic upper bound has been provided for these distributions, to understand which parameter is varying in this minimax result remains a hard issue.
					
					Consequently, this subsection is devoted to obtaining a minimax lower bound on the excess risk over the set of distributions satisfying the margin condition defined in Definition \ref{margincondition}, in which some parameters are fixed. Throughout this subsection, $\hat{\c}_n$ will denote an empirically designed codebook, that is a map from $(\mathbb{R}^d)^n$ to $(\mathbb{R}^d)^k$. Let $k$ be an integer such that $k \geq 3$, and $M>0$. For simplicity, $k$ is assumed to be divisible by $3$. Let us introduce the following quantities:
											\begin{align*}
											\left \{
											\begin{array}{@{}ccc}
											m&=&\frac{2k}{3} \\
											\Delta&=&\frac{15M}{96m^{1/d}}.
											\end{array}
											\right.
											\end{align*}
											
											To focus on the dependency on the separation factor $\varepsilon$, the quantities involved in Definition \ref{margincondition} are fixed as:
											\begin{align}\label{conditionmargeminimax}
											\left \{
											\begin{array}{@{}ccc}
											B&=&\Delta \\
											r_0&=& \frac{7 \Delta}{16}\\
											p_{min} &\geq& \frac{1}{2k}.
											\end{array}
											\right .
											\end{align}
     Denote by $\mathcal{D}(\varepsilon)$ the set of probability distributions which are $\varepsilon$-separated, and which satisfies a margin condition with parameters defined in \eqref{conditionmargeminimax}. The minimax result is the following:
		\begin{prop}\label{minimax}
		Assume that $k \geq 3$. Then, for any empirically designed codebook,
		\begin{align*}
		\mathbb{E} \sup_{P \in \mathcal{D}(c_1/\sqrt{n})} \ell(\hat{\c}_n,\c^*) \geq c_0 M^2 \frac{\sqrt{k^{1-\frac{4}{d}}}}{\sqrt{n}},
		\end{align*}
		where $c_0$ is an absolute constant, and
		\[
		c_1 = \frac{(15M)^2}{4(96m^{\frac{1}{4}+\frac{1}{d}})^2}.
		\]
		\end{prop}
		    Proposition \ref{minimax} can be thought of as an extension of Theorem 1 in \cite{Bartlett98}. This minimax lower bound has to be compared to the upper risk bound obtained in Theorem \ref{mainresult} for the empirical risk minimizer $\hat{\c}_n$ over the set of distributions $\mathcal{D}(c_1/\sqrt{n})$. To be more precise, Theorem \ref{mainresult} ensures that, provided that $n$ is large enough, 
				\[
				\mathbb{E} \sup_{P \in \mathcal{D}(c_1/\sqrt{n})} \leq \frac{g(k,d,M)}{\sqrt{n}},
				\]
				where $g(k,d,M)$ depends only on $k$, $d$ and $M$. In other words, the dependency of the upper bounds stated in Theorem \ref{mainresult} on $\varepsilon$ turns out to be sharp whenever $\varepsilon \sim n^{-\frac{1}{2}}$. Unfortunately, Proposition \ref{minimax} can not be easily extended to the case where $\varepsilon \sim n^{-\alpha}$, with $0<\alpha<1/2$. Consequently an open question is whether the upper bounds stated in Theorem \ref{mainresult} remains accurate with respect to $\varepsilon$ in this case.

       \subsection{Quasi-Gaussian mixture example}\label{Gaussianmixture}

                The aim of this subsection is to illustrate the results offered in Section \ref{Results} with Gaussian mixtures in dimension $d=2$. The Gaussian mixture model is a typical and well-defined clustering example. However we will not deal with the clustering issue but rather with its theoretical background.

                In general, a Gaussian mixture distribution $\tilde{P}$ is defined by its density
                 \[
                 \tilde{f}(x) = \sum_{i=1}^{\tilde{k}}{\frac{\theta_i}{2 \pi  \sqrt{ \left | \Sigma_i \right | }}e^{-\frac{1}{2}(x-m_i)^t \Sigma_i^{-1} (x-m_i)}},
                 \]
                where $\tilde{k}$ denotes the number of component of the mixture, and the $\theta_i$'s denote the weights of the mixture, which satisfy $\sum_{i=1}^{k}{\theta_i} = 1$. Moreover, the $m_i$'s denote the means of the mixture, so that $m_i$ $\in$ $ \mathbb{R}^2$, and the $\Sigma_i$'s are the $2\times 2$ variance matrices of the components.

                We restrict ourselves to the case where the number of components $\tilde{k}$ is known, and match the size $k$ of the codebooks. To ease the calculation, we make the additional assumption that every component has the same diagonal variance matrix $\Sigma_i = \sigma^2 I_2$. Note that a  similar result to Proposition \ref{Gaussianboundary} can be derived for distributions with different variance matrices $\Sigma_i$, at the cost of more computing.

                 Since the support of a Gaussian random variable is not bounded, we define the ``quasi-Gaussian" mixture model as follows, truncating each Gaussian component.
                 Let the density $f$ of the distribution $P$ be defined by
                 \[
                 f(x) = \sum_{i=1}^{k}{\frac{\theta_i}{ 2 \pi \sigma^2 N_i} e^{-\frac{\|x-m_i\|^2}{2 \sigma^2}}}\mathbbm{1}_{\mathcal{B}(0,M)},
                 \]
 where $N_i$ denotes a normalization constant for each Gaussian variable.

 To ensure this model to be close to the Gaussian mixture model, we assume that there exists a constant $\varepsilon \in \left [0, 1\right ]$ such that, for $i=1, \hdots, k$, $N_i \geq 1 - \varepsilon$.

 Denote by $\tilde{B} = {\inf_{i \neq j}}{\|m_i - m_j\|}$ the smallest possible distance between two different means of the mixture. To avoid boundary issues we assume that, for all $i = 1, \hdots, k$, $\mathcal{B}(m_i, \tilde{B}/3) \subset \mathcal{B}(0,M)$.
 
 It is worth noticing that the two assumptions $N_i \geq 1 - \varepsilon$ and $\mathcal{B}(m_i, \tilde{B}/3) \subset \mathcal{B}(0,M)$ can easily be satisfied as soon as $M$ is chosen large enough. For such a model, Proposition \ref{Gaussianboundary} offers a sufficient condition for $P$ to satisfy a margin condition.

                 \begin{prop}\label{Gaussianboundary}
                 Let $\theta_{min} = \min_{i=1,\hdots,k}{\theta_i}$, and $\theta_{max}=\max_{i=1,\hdots,k}{\theta_i}$. Assume that
                 \begin{align}\label{gaussiancondition}
                 \frac{\theta_{min}}{\theta_{max}} \geq \max {\left(\frac{2048 k \sigma^2}{(1-\varepsilon) \tilde{B}^2(1 - e^{-\tilde{B}^2/{2048\sigma^2}})},  \frac{2048 k^2 M^3}{(1 - \varepsilon)7\sigma^2 \tilde{B}(e^{\tilde{B}^2/{32\sigma^2}}-1)}\right ) }.
                 \end{align}
                 Then $P$ satisfies a margin condition with radius $\frac{\tilde{B}}{8}$.
                 \end{prop}
                  The condition \eqref{gaussiancondition} can be decomposed as follows. If
                  \[
                  \frac{\theta_{min}}{\theta_{max}} \geq \frac{2048 k \sigma^2}{(1-\varepsilon) \tilde{B}^2(1 - e^{-\tilde{B}^2/{2048\sigma^2}})},
                  \]
                   then the optimal codebook $\c^*$ is close to the vector of means of the mixture $\mathbf{m} = (m_1, \hdots, m_k)$. Therefore, it is possible to locate the critical region associated with the optimal codebook $\c^*$, and to derive an upper bound on the weight function defined in Definition \ref{margincondition}. This leads to the second term of the maximum in  \eqref{gaussiancondition}.

                  This condition can be interpreted as a condition on the polarization of the mixture. A favorable case for vector quantization seems to be when the poles of the mixtures are well separated, which is equivalent to $\sigma$ is small compared to $\tilde{B}$, when considering Gaussian mixtures. Proposition \ref{Gaussianboundary} gives details on how $\sigma$ has to be small compared to $\tilde{B}$, in order to satisfy the requirements of Proposition \ref{lienconditionmarge}. This ensures that the loss $\ell(\hat{\c}_n,\c^*)$ reaches an improved convergence rate of $1/n$.

                  It may be noticed that Proposition \ref{Gaussianboundary} offers almost the same condition than Proposition 4.2 in \cite{Levrard12}. In fact, since the Gaussian mixture distributions have a continuous density, making use of \eqref{continuousdensity} in Example 3 ensures that the margin condition for Gaussian mixtures is equivalent to a bound on the density over the critical region. 

                  It is important to note that this result is valid when $k$ is known and match exactly the number of components of the mixture. When the number of code points $k$ is different from the number of components $\tilde{k}$ of the mixture, we have no general idea of where the optimal code points can be located.

                   Moreover, suppose that there exists only one optimal codebook $\c^*$, up to reindexing, and that we are able to locate this optimal codebook $\c^*$. As mentioned in Proposition \ref{lienconditionmarge}, the key quantity is in fact $B = \inf_{i\neq j}\|c^*_i - c^*_j\|$. In the case where $\tilde{k} \neq k$, there is no simple relation between $\tilde{B}$ and $B$.  Consequently, a condition like in Proposition \ref{Gaussianboundary} could not involve the natural parameter of the mixture $\tilde{B}$.

                   It is also worth pointing out that there exist cases where the set of optimal codebooks is not finite. For example, assume that $P$ is a truncated rotationally symmetric Gaussian distribution, and $k=2$. Since every rotation of an optimal codebook leads to another optimal codebook, there exists an infinite set of optimal codebooks. Since, in this case, $N^* = \mathcal{B}(0,M)$, obviously $P$ does not satisfy a margin condition.

         \section{Proofs}\label{Proofs}
         
         \subsection{Proof of Proposition \ref{lienconditionmarge}}\label{Proof of Proposition {lienconditionmarge}} 
         The proof of Proposition \ref{lienconditionmarge} is based on the following lemma.
         \begin{lem}\label{boundarycloseness}
      Let $x \in V_i(\c^*) \cap V_j(\c)$, for $i \neq j$. Then  
      \begin{align}  
      \left | \left\langle x - \frac{c_i + c_j}{2}, c_i - c_j\right\rangle \right | & \leq   4 \sqrt{2}M \| \c - \c^* \| \label{Vor1} , \\
      d(x, \partial V_i(\c^*)) & \leq  \frac{4 \sqrt{2} M}{B} \| \c - \c^* \| \label{Vor2}.
      \end{align}
\end{lem} 

    The two statements of Lemma \ref{boundarycloseness} emphasize the fact that, provided that $\c$ and $\c^*$ are quite similar, the areas on which the label may differ with respect to $\c$ and $\c^*$ should be close to the boundary of Voronoi diagrams. This idea is mentioned in the proof of Corollary 1 in \cite{Antos04}. Nevertheless we provide here a simpler proof.
    
    \begin{proof}[Proof of Lemma \ref{boundarycloseness}] 
    Let $x \in V_i(\c^*) \cap V_j(\c)$, then$ \| x - c_j \|^2 \leq \| x - c_i \|^2$, which leads to $\left\langle c_i - c_j, x - \frac{c_i +c_j}{2} \right\rangle \leq 0$.  
    Since $\| x - c_i^* \| \leq \| x - c_j^* \|$, we may write
    \[
    \| x - c_i \| \leq \| x - c_j \| + \| c_i - c_i^*\| + \|c_j - c_j^*\|.
    \]
    Taking square on both sides leads to
    \begin{align*}
    \| x - c_i \|^2 - \| x-c_j\|^2 & \leq \begin{multlined}[t] 2 \| x - c_j \| ( \| c_i - c_i^*\| + \|c_j - c_j^*\| ) \\ + \left ( \| c_i - c_i^*\| + \|c_j - c_j^*\| \right )^2 \end{multlined} \\
                                   & \leq 8M  ( \| c_i - c_i^*\| + \|c_j - c_j^*\| ) \\
                                   & \leq 8 \sqrt{2}M \| \c - \c^* \|.
    \end{align*}
             Since $ \| x - c_i \|^2 - \| x-c_j\|^2 = 2 \left\langle x - \frac{c_i + c_j}{2}, c_i - c_j\right\rangle $, \eqref{Vor1} is proved.
             
             To prove \eqref{Vor2}, remark that, since $x \in V_i(\c^*)$, $d(x,V_i(\c^*)) \leq d(x,h^*_{i,j})$, where $h^*_{i,j}$ is the hyperplane defined by $\left \{x \in \mathcal{B}(0,M)| \| x - c_i^*\| = \|x - c_j^*\| \right \}$. Using quite simple geometric arguments, we deduce that
             \[
             d(x,h^*_{i,j}) = \left |\left\langle x - \frac{c_i^* + c_j^*}{2}, \frac{c_i^* - c_j^*}{\left \| c_i^* - c_j^* \right \|}\right\rangle \right |.
             \]
             The same arguments as in the proof of \eqref{Vor1} guarantee that
             \[
             \begin{aligned}
            \left |\left\langle x - \frac{c_i^* + c_j^*}{2}, \frac{c_i^* - c_j^*}{\left \| c_i^* - c_j^* \right \|}\right \rangle \right | & = \left\langle x - \frac{c_i^* + c_j^*}{2}, \frac{c_i^* - c_j^*}{\left \| c_i^* - c_j^* \right \|}\right\rangle \\
              & \leq \frac{4 \sqrt{2}M}{B} \| \c - \c^* \|.
              \end{aligned}
              \]
\end{proof}

Equipped with Lemma \ref{boundarycloseness}, we are in a position to prove Proposition \ref{lienconditionmarge}. 
Let $\c \in \mathcal{M}$, and $(W_1(\c), \hdots, W_k(\c))$  be a Voronoi partition associated to $\c$, as defined in Section \ref{Notation}. Then $\ell(\c,\c^*)$ can be decomposed as follows:
        \[
        \begin{aligned}
        P \gamma(\c,.) &=\sum_{i=1}^{k}{P(\| x - c_i \|^2 \mathbbm{1}_{W_i(\c)})} \\
        & = \sum_{i=1}^{k}{P(\| x - c_i \|^2 \mathbbm{1}_{V_i(\c^*)})} + \sum_{i=1}^{k}{P(\|x-c_i\|^2(\mathbbm{1}_{W_i(\c)}-\mathbbm{1}_{V_i(\c^*)}))}.
        \end{aligned}
        \]
        
        Since, for all $i = 1,\hdots k$, $P(x\mathbbm{1}_{V_i(\c^*)}(x)) = P(V_i(\c^*))c_i^*$ (centroid condition), we may write
        \[
        P(\| x - c_i \|^2 \mathbbm{1}_{V_i(\c^*)}) = P(V_i(\c^*)) \|c_i - c_i^*\|^2 + P( \| x - c_i^* \|^2 \mathbbm{1}_{V_i(\c^*)}),
        \]
        from which we deduce
        \[
        P \gamma(\c,.) = P \gamma(\c^*,.) + \sum_{i=1}^{k}{P(V_i(\c^*)) \|c_i - c_i^*\|^2} + \sum_{i=1}^{k}{P(\|x-c_i\|^2(\mathbbm{1}_{W_i(\c)}-\mathbbm{1}_{V_i(\c^*)}))},
        \]
        which leads to
        \[
        \ell(\c,\c^*) \geq p_{min} \|\c - \c^*\|^2 + \sum_{i=1}^{k}{ \sum_{j \neq i} {P\left ((\|x-c_j\|^2 - \|x-c_i\|^2)\mathbbm{1}_{V_i(\c^*)\cap W_j(\c)} \right )}}.
        \]
        
        Since $x\in W_j(\c) \subset V_j(\c)$, $\|x-c_j\|^2 - \|x-c_i\|^2 \leq 0$. Thus it remains to bound from above
        \[
 \sum_{i=1}^{k}{ \sum_{j \neq i} {P\left ((\|x-c_i\|^2 - \|x-c_j\|^2)\mathbbm{1}_{V_i(\c^*)\cap W_j(\c)} \right )}}.
        \]
Noticing that
\[
\|x-c_i\|^2 - \|x-c_j\|^2 = 2 \left\langle c_j -c_i, x_{i,j} - \frac{c_i+c_j}{2} \right\rangle,
\]
and using Lemma \ref{boundarycloseness}, we get
\[
 \sum_{i=1}^{k}{P(\|x-c_i\|^2(\mathbbm{1}_{W_i(\c)}-\mathbbm{1}_{V_i(\c^*)}))} \geq - 8 \sqrt{2}M \| \c - \c^* \| p \left ( \frac{4 \sqrt{2}M}{B}\| \c - \c^* \| \right ).
 \]
 
 Consequently, if $P$ satisfies \eqref{majorationkappa}, then, if $\| \c - \c^* \| \leq  \frac{B r_0}{4\sqrt{2}M}$, 
 \[
 \ell(\c,\c^*) \geq \frac{p_{min}}{2}\| \c - \c^*\|^2.
 \] 
 Now turn to the case where $\| \c - \c^*(\c) \| \geq  \frac{B r_0}{4\sqrt{2}M}$. Since the support of $P$ is included on $\mathcal{B}(0,M)$, the function $\c \longmapsto P\gamma(\c,.)$ is continuous , its minimum on $(\mathbb{R}^d)^k \cap \left ( \bigcup_{\c^* \in \mathcal{M}}{\mathcal{B}(0,M)} \right )^{c}$ is attained. Such a minimizer is a local minimizer, or is at the boundary $\| \c - \c^*(\c) \| = \frac{B r_0}{4\sqrt{2}M}$. Hence we deduce 
 \[
 \begin{aligned}
 \ell(\c,\c^*) & \geq \varepsilon \wedge \frac{p_{min} B r_0 ^2}{64 M^2} \\
               & \geq \left (\varepsilon \wedge \frac{p_{min} B r_0 ^2}{64 M^2} \right ) \frac{\| \c - \c^* \|^2}{4kM^2}.
 \end{aligned}
 \]
 This proves Proposition \ref{lienconditionmarge}

\subsection{Proof of Theorem \ref{mainresult}}
      
      Throughout this subsection $P$ is assumed to satisfy a margin condition with radius $r_0$, and to be $\varepsilon$-separated.
      A non decreasing map $\Phi:\mathbb{R} \rightarrow \mathbb{R}^+$ is called sub-$\alpha$ if 
			$x \mapsto \frac{\Phi(x)}{x^{\alpha}}$ is non increasing.
			
      The following localization theorem, derived from Theorem 6.1 in \cite{Blanchard08}, is the main argument of our proof.

\begin{thm}\label{localization}
                 
               Let $\mathcal{F}$ be a class of bounded measurable functions such that there exist $b>0$ and $\omega: \mathcal{F} \longrightarrow \mathbb{R}^+$ satisfying
               \begin{itemize}
               \item[$(i)$] $\forall f \in \mathcal{F} \quad \left\|f\right\|_{\infty} \leq b$,
               \item[$(ii)$] $\forall f \in \mathcal{F} \quad  \Var(f) \leq \omega(f)$. 
               \end{itemize}
               Let $K$ be a positive constant, $\Phi$ a sub-$\alpha$ function, $\alpha \in \left[1/2,1\right[$. Then there exists a constant $C(\alpha)$ such that, if $D$ is a constant satisfying $D \leq 6K C(\alpha)$, and $r^*$ is the unique solution of the equation $\Phi(r) = r/D$, the following holds.
                Assume that
               \[
               \forall r \geq r^* \qquad \mathbb{E} \left ( \sup_{\omega(f)\leq r}{\left |(P-P_n)f \right |} \right ) \leq \Phi(r).
               \]
               Then, for all $x>0$, with probability larger than $1-e^{-x}$,
               \[
               \forall f \in \mathcal{F} \quad Pf - P_n f \leq K^{-1} \left ( \omega(f) + \left (\frac{6K C(\alpha)}{D} \right )^{\frac{1}{1-\alpha}}r^* + \frac{(9K^2+16Kb)x}{4n} \right ).
               \] 
               \end{thm}
               
               A proof of Theorem \ref{localization} is given in Section 5.3 of \cite{Levrard12}. Notice that an explicit calculation of $C(\alpha)$ is given by $C(\alpha) = \underset{x>1}{\inf} {\left ( 1 + x^\alpha \left ( \frac{1}{2} + \frac{1}{x^{1-\alpha}-1} \right ) \right )}$.
               
               \subsubsection{Proof of (\ref{firstinequality})}
               
               The proof of \eqref{firstinequality} follows from the combination of Proposition \ref{lienconditionmarge} and a direct application of Theorem \ref{localization}. To be more precise, let $\mathcal{F}_1$ denote the set
               \[
               \mathcal{F}_1 = \left \{ \gamma(\c,.) - \gamma(\c^*(\c),.) | \quad \c \in \mathcal{B}(0,M)^k \right \}.
               \]
               Since, for all $i$ $\in$ $\left \{1, \hdots, k \right \}$, 
               \[
               \left | \| x - c_i \|^2 - \|x - c_i^*(\c)\|^2 \right | \leq 4M \| c_i - c_i^*(\c) \|,
               \]
               it follows that, for every $f$ $\in$ $\mathcal{F}_1$,
               \[
               \left \{
               \begin{array}{@{}ccl}
               \|f\|_{\infty} & \leq & 8 M^2 \\
               \Var_P(f) & \leq & 16 M^2 \| \c - \c^*(\c) \|^2.
               \end{array}
               \right .
               \]
               
               Define $\omega_1(f) = 16 M^2 \| \c - \c^*(\c) \|^2$. It remains to bound from above the complexity term. This is done in the following proposition, derived from the proof of Theorem 1 in \cite{Chichi13}.
               \begin{prop}\label{chainagechichignoud}
							One has
               \begin{align}\label{definitiondelta1}
               \mathbb{E} \sup_{f \in \mathcal{F}_1, \omega_1(f) \leq \delta}{\left | (P-P_n) f \right |} 
               \leq \frac{(2\sqrt{2} + 64)\sqrt{kd}}{\sqrt{n}} \left ( \sqrt{\log(4 |\mathcal{M}| \sqrt{kd})} +1 \right ) \sqrt{\delta}.
               \end{align}
               \end{prop}
               The proof of Proposition \ref{chainagechichignoud} derives from classical chaining arguments, and is given in Section \ref{ProofofPropositionchainagechichignoud}. Let $\Phi_1$ be defined as the right-hand side of \eqref{definitiondelta1}. Observing that $\Phi_1(\delta)$ takes the form $\Phi_1(\delta) = \Xi_1 \sqrt{\delta/n}$, the solution $\delta_1^*$ of the equation $\Phi_1(\delta) = \delta/D$ may be written, for any $D >0$,
               \[
               \delta_1^*= \frac{D^2 \Xi_1^2}{n}.
               \]
               Let $K>0$ and choose $D = 6 K C(1/2)$. Applying Theorem \ref{localization} to $\mathcal{F}_1$ leads to, with probability larger than $1-e^{-x}$,
               \begin{multline*}
               (P-P_n)(\gamma(\mathbf{c},.) - \gamma(\mathbf{c}^*(\c),.)) \leq K^{-1} 16 M^2 \|\mathbf{c}-\mathbf{c}^*(\c)\|^2 \\ + \frac{36K C(1/2)^2\Xi_1^2}{n} + \frac{9K+128M^2}{4n}x.
               \end{multline*}
               Introducing the inequality $ \kappa_0 \ell(\c,\c^*) \geq \| \c - \c^*(\c) \|^2$ provided by Proposition \ref{lienconditionmarge}, choosing $K = 32 M^2 \kappa_0$ and taking into account that $C(1/2) \leq 4$ leads to \eqref{firstinequality}.
               
               \subsubsection{Proof of  (\ref{secondinequality})} 
               
                The proof of \eqref{secondinequality} also relies on an application of Theorem \ref{localization}. Let the loss of $\hat{\c}_n$ be decomposed as follows,
\begin{align}\label{eqdep}
P(\gamma(\hat{\c}_n) - \gamma(\c^*)) & \leq (P-P_n)(\gamma(\hat{\c}_n) - \gamma(\c^*)) \notag\\
&\begin{multlined} \leq (P-P_n)\left\langle \hat{\c}_n - \c^*(\hat{\c}_n), \Delta(\c^*(\hat{\c}_n),.) \right\rangle \\
 + (P-P_n)\|\hat{\c}_n - \c^*(\hat{c}_n)\|R(\hat{\c}_n,\c^*(\hat{c}_n),.),
 \end{multlined} 
\end{align}
where 
\[
                \Delta(\mathbf{c}^*,x)  = -2( (x-c^*_1)\mathbbm{1}_{V_1(\c^*)},...,(x-c^*_k)\mathbbm{1}_{V_k(\c^*)} ),
\]                
                and
\begin{multline*}                
                R(\c,\c^*,x) = \sum_{i,j =1, \hdots, k} {\mathbbm{1}_{V_i(\c^*)\cap W_j(\c)} \| \c - \c^* \|^{-1} \left [ \vphantom{\left\langle x - \frac{c_i + c_j}{2}, c_i - c_j\right\rangle} \| c_i - c_i^*\|^2 \right .}   \\ \left . + 2 \left\langle x - \frac{c_i + c_j}{2}, c_i - c_j\right\rangle \right ],
               \end{multline*}
where we recall that $W_i(\c)$ denotes an element of a Voronoi partition, such that $\bar{W}_i(\c) \subset V_i(\c)$. The proof of \eqref{secondinequality} consists in applying Theorem \ref{localization} to the two terms in the right-hand side of \eqref{eqdep}.

   The first term on the right-hand side of \eqref{eqdep} may be thought of as the dominant term in the decomposition of the loss. Define 
	\[
             \mathcal{F}_2 = \left \{ \left\langle \c - \c^*(\c), \Delta(\c^*(\c),.) \right\rangle) | \quad \c \in \mathcal{B}(0,M) \right \}.
             \]
	In order to apply Theorem \ref{localization}, the following lemmas are needed.
	\begin{lem} \label{controlepremierterme}
            Let $f \in \mathcal{F}_2$, then
            \[
            \left \{
            \begin{array}{@{}ccl}
            \|f \|_{\infty}& \leq & 8M \\
            \Var_P(f) &\leq& 4 \|\c - \c^*(\c)\|^2 R(\c^*).
            \end{array}
            \right.
            \]
            \end{lem}
            \begin{proof}[Proof of Lemma \ref{controlepremierterme}]
            Elementary calculation shows that
            \begin{align*}
\Var(\left\langle \c - \c^*, \Delta(\c^*,.) \right\rangle) = & P(\left\langle \c - \c^*, \Delta(\c^*,.) \right\rangle)^2 - (P(\left\langle \c - \c^*, \Delta(\c^*,.) \right\rangle))^2 \\
                                = & \sum_{i=1}^{k}{P\left [ \left\langle c_i - c_i^*, -2(x-c_i^*) \right\rangle ^2 \mathbbm{1}_{V_i(\c^*)}(x) \right ]} \\
                                \leq & 4 \|\c - \c^*\|^2 R(\c^*).
                                \end{align*}                               
                                \end{proof}
                                
                                Let $\omega_2(f)$ be defined as $4 \|\c - \c^*(\c)\|^2 R(\c^*)$. It remains to bound from above the expectation of the maximum deviation between $P$ and $P_n$ over the set $\mathcal{F}_2$. 
                                \begin{lem}\label{complexitepremierterme}
																One has
                                \begin{align}\label{definitionphi2}
                                             \mathbb{E} \sup_{f \in \mathcal{F}_2, \omega_2(f) \leq \delta}{\left |(P-P_n)f\right|}  \leq \frac{2 |\mathcal{M}|}{\sqrt{n}} \sqrt{\delta}
                                             \end{align}
                                             \end{lem}
                                             \begin{proof}
                                               This proof is inspired from the proof of Lemma 4.3 in \cite{Biau08}. The first step is the following
                               \begin{multline*}
       \mathbb{E} \sup_{f \in \mathcal{F}_2, \omega_2(f) \leq \delta}{\left |(P-P_n)f\right|}  \\ \leq \mathbb{E} \sup_{\c^* \in \mathcal{M}, \|\c-\c^*\| \leq \sqrt{\delta/4R(\c^*)}} {\left |(P-P_n)\left\langle \c - \c^*, \Delta(\c^*,.) \right\rangle \right |.
       }        
       \end{multline*}
            For a general function $h(Z)$ depending on a random map $Z$, we denote by $\mathbb{E}_Z h$ the expectation of $h$ taken with respect to $Z$. Introducing some Rademacher independent random variables $\sigma_i$ and using a symmetrization inequality such as in Section 2.2 of \cite{Koltchinskii04} leads to
             \begin{multline*}
\mathbb{E}\sup_{\|\c-\c^*\| \leq \sqrt{\delta/4R(\c^*)}, \c^* \in \mathcal{M}}{\left |(P - P_n)\left\langle \c - \c ^*, \Delta(\c^*,.) \right\rangle \right |} \\ 
   \begin{aligned}
   & \leq 2 \mathbb{E}_X \mathbb{E}_\sigma \sup_{\|\c-\c^*\| \leq \sqrt{\delta/4R(\c^*)}, \c^* \in \mathcal{M}} { \left\langle \c - \c ^*,\frac{1}{n} \sum_{i=1}^{n}{\sigma_i \Delta(\c^*,X_i)} \right\rangle} \\
     &\leq \sqrt{\delta/4R(\c^*)} 2 \mathbb{E}_X \mathbb{E}_\sigma \sup_{\c^* \in \mathcal{M}}{\left \| \frac{1}{n} \sum_{i=1}^{n}{\sigma_i \Delta(\c^*,X_i)} \right \| },
    \end{aligned}  
    \end{multline*}
           using Cauchy-Schwarz inequality. Eventually,
           \[
           \begin{aligned}
           \mathbb{E}_X \mathbb{E}_\sigma \sup_{\c^* \in \mathcal{M}}{\left \| \frac{1}{n} \sum_{i=1}^{n}{\sigma_i \Delta(\c^*,X_i)} \right \| } & \leq \sum_{\c^* \in \mathcal{M}}{\mathbb{E}_X \mathbb{E}_\sigma \left \| \frac{1}{n} \sum_{i=1}^{n}{\sigma_i \Delta(\c^*,X_i)} \right \|} \\
           & \leq \sum_{\c^* \in \mathcal{M}}{\sqrt{\mathbb{E}_X \mathbb{E}_\sigma \left \| \frac{1}{n} \sum_{i=1}^{n}{\sigma_i \Delta(\c^*,X_i)} \right \|^2}}\\
           & \leq \sum_{\c^* \in \mathcal{M}}{\frac{1}{\sqrt{n}} \sqrt{\mathbb{E}_X \left \|\Delta(\c^*,X) \right \|^2}} \\
           & \leq \frac{2|\mathcal{M}| \sqrt{R(\c^*)}}{\sqrt{n}},
           \end{aligned}
           \]
           where Jensen's inequality has been used to obtain the second line. This gives the desired result.
            \end{proof}
						
		The contribution of the first term in the right-hand side of \eqref{eqdep} is described by the following proposition. 
   \begin{prop}\label{concentrationpremierterme}
   Let $K_2$ be a positive constant and $x>0$. Then, with probability larger than $1 - e^{-x}$
   \begin{multline*}
            (P-P_n)  \left\langle  \hat{\c}_n - \c ^*(\hat{\c}_n), \Delta(\c^*(\hat{\c}_n),.) \right\rangle \leq K_2^{-1} \left [ \vphantom{\frac{48^2 |\mathcal{M}|^2K_2^2}{n} + \frac{9 K_2^2 + 128M^2 \sqrt{k} K_2}{4n}x} 4 R(\c^*) \| \hat{\c}_n - \c^*(\hat{\c}_n) \|^2 \right .  \\
            + \left . \frac{48^2 |\mathcal{M}|^2K_2^2}{n} + \frac{9 K_2^2 + 128M^2 \sqrt{k} K_2}{4n}x\right].
            \end{multline*}
           \end{prop}
     The proof follows from a direct application of Theorem \ref{localization} to the set $\mathcal{F}_2$, replacing the value $C(1/2)$ with $4$ to ease the calculation.

The second term in the right-hand side of \eqref{eqdep} may be thought of as a residual term. Deriving sharper bounds on this term requires more accurate chaining techniques, as exposed below. Define 
\[
\mathcal{F}_3 = \left \{ \|\c-\c^*(\c)\| R(\c, \c^*(\c),.)|\c \in \mathcal{B}(0,M) \right \}.
 \]
In order to apply Theorem \ref{localization}, the following intermediate results are needed.
\begin{lem}\label{controlesecondterme}
        Let $f$ $\in$ $\mathcal{F}_3$, then
        \[
            \left \{
            \begin{array}{@{}ccl}
            \|f\|_{\infty} & \leq & 2M\sqrt{k} C_{\infty}\\
            \Var_{P}\left (  f \right ) & \leq& C_{\infty}^2 \| \c - \c^*(\c) \| ^2, 
            \end{array}
            \right.
            \]
            with
            \[
            C_\infty = (2\sqrt{k} + 8\sqrt{2})M.
            \]
            \end{lem}
            
            \begin{proof}[Proof of Lemma \ref{controlesecondterme}]
            The proof of Lemma \ref{controlesecondterme} follows from a bound on $R(\c,\c^*(\c),x)$, namely
         \begin{align*}
         |R(\c,\c^*,x)| & \begin{multlined}[t] = \|\c-\c^*\|^{-1} \left | \sum_{i,j}{\mathbbm{1}_{V_i(\c^*) \cap W_j(\c)}\left ( \vphantom{\left\langle x - \frac{c_i + c_j}{2}, c_i - c_j\right\rangle}\| c_i - c_i^*\|^2 \right .} \right .\\ +  \left .\left . 2 \left\langle x - \frac{c_i + c_j}{2}, c_i - c_j\right\rangle \right ) \right | \end{multlined}\\
                        & \begin{multlined} \leq \|\c-\c^*\|^{-1} \left [ \vphantom{\left\langle x - \frac{c_i + c_j}{2}, c_i - c_j\right\rangle} \sum_{i}{\| c_i - c_i^*\|^2 \mathbbm{1}_{V_i(\c^*)}} \right . \\ \left . + \sum_{i \neq j}{2\left | \left\langle x - \frac{c_i + c_j}{2}, c_i - c_j\right\rangle \right |\mathbbm{1}_{V_i(\c^*) \cap W_j(\c)} } \right ].\end{multlined} 
         \end{align*} 
      Since, for all $j$ in $\left \{1, \hdots, k \right \}$, $W_j(\c) \subset V_j(\c)$, applying Lemma \ref{boundarycloseness} leads to               
           \begin{align}
           |R(\c,\c^*,x)|  & \leq \|\c-\c^*\|^{-1} \left [ \| \c - \c^* \|^2 + 8 \sqrt{2}M \| \c - \c^* \| \mathbbm{1}_{N^*(\frac{4\sqrt{2}M}{B}\| \c - \c^* \|)} \right ] \notag \\ 
                        & \leq \| \c - \c^* \| + 8 \sqrt{2}M \mathbbm{1}_{N^*(\frac{4\sqrt{2}M}{B}\| \c - \c^* \|)}:= F_{\| \c - \c^* \|}(x) \label{definitionF}.
         \end{align}         
         Elementary calculations show that, for any $\delta >0$,
         \[
         \|F_{\delta}\|_{\infty}  \leq (2\sqrt{k} + 8\sqrt{2})M = C_\infty, 
         \]
         from which we deduce the desired upper bounds on $\Var_{P}(f)$ and $\|f\|_\infty$, for $f$ in $\mathcal{F}_3$.
         \end{proof}
         
         Let $\omega_3(f)$ be defined as $C_{\infty}^2 \| \c - \c^*(\c) \| ^2$. The complexity term associated with the class of functions $\mathcal{F}_3$ can be bounded as follows.
         \begin{prop}\label{complexitesecondterme}
         \[
         \mathbb{E} \sup_{f \in \mathcal{F}_3, \omega_3(f) \leq \delta} |(P - P_n)f|  \leq \frac{8 Q(k,d)}{\sqrt{ C_\infty n}} \sqrt{C_2(\sqrt{\delta} /C_{\infty}) \sqrt{\delta}},
         \]
         where
         \[
         \left \{
         \begin{array}{@{}ccc}
            C_2(r) & = & r + 8 \sqrt{2} M p\left(\frac{4 \sqrt{2}M}{B} r\right ) \\           
             Q(k,d) &= & 8 \sqrt{K_0P(k,d)\log(k^2(4k-2))} \\
             P(k,d)&=&k^2(2(k-1)(d+1) +24(3d+4)),
         \end{array}    
         \right .
         \]
         and $K_0$ is an absolute constant.
         \end{prop}
         
         The proof of Proposition \ref{complexitesecondterme} is based on a result of Mendelson and Vershynin in \cite{Mendelson03} and its application to a more accurate version of Dudley's integral. For clarity, the proof is postponed to Section \ref{complexitebourrin}.
				 Since $P$ satisfies a margin condition with parameters $(r_0,\kappa)$, with $\kappa \leq \frac{Bp_{min}}{128 M^2}$, considering the two cases $4\sqrt{2}Mr/B \leq r_0$ and $4\sqrt{2}Mr/B \geq r_0$ yields that
         \[
         8\sqrt{2}Mp(\frac{4\sqrt{2}M}{B}r) \leq \frac{64 M^2}{B r_0} r,
         \]
         for $r \geq 0$. Using this inequality to bound $C_2$ from above in Proposition \ref{complexitesecondterme} leads to the following complexity result
         \begin{align}\label{definitionphi3}
         \mathbb{E} \sup_{f \in \mathcal{F}_3, \omega_3(f) \leq \delta} |(P - P_n)f| \leq \frac{\Xi_3}{\sqrt{n}}\delta^{\frac{3}{4}},
         \end{align}
         where 
         \[
         \Xi_3 = \frac{8M Q(k,d)}{C_\infty \sqrt{Br_0}}.
         \]
         Let $\Phi_3$ be defined as $\frac{\Xi_3}{\sqrt{n}}\delta^{\frac{3}{4}}$. Remark that $\Phi_3$ is a sub-$3/4$ function. Consequently, for any $D>0$, the solution of the equation $\Phi_3(\delta)=\delta/D$ is
         \[
         \delta_3^* = \frac{(D \Xi_3)^4}{n^2}.
         \]
         Choosing $K_3>0$, and $D=6 K_3C(3/4)$ in Theorem \ref{localization} and taking into account that $C(3/4) \leq 10$ leads to the following proposition.
        
        \begin{prop}\label{concentrationsecondterme}
        Let $K_3 >0$. Then, with probability larger than $1-e^{-x}$, 
        \begin{multline*}
            (P-P_n)\left ( \| \hat{\c}_n - \c^*(\hat{\c}_n) \| R(\hat{\c}_n, \c^*(\hat{\c}_n),.) \right ) \leq 
            K_3^{-1} \left [ \vphantom{\frac{60^4K_3^4C(k,d,B,M)^4}{n^2}} C_\infty^2 \| \hat{\c}_n - \c^*(\hat{\c}_n) \|^2 \right . \\
            \left .+ \frac{60^4 K_3^4 \Xi_3^4}{n^2} + \frac{9 K_3^2 + 32 M \sqrt{k} C_\infty K_3}{4n}x \right ],
            \end{multline*}
            with $\Xi_3= \frac{8MQ(k,d)}{C_\infty \sqrt{Br_0}}$, and $Q$ is a function composed of products of square roots of polynomial functions in $k$, $d$, and $\log(k)$.
        \end{prop}

         We are now in position to prove \eqref{secondinequality}. Proposition \ref{lienconditionmarge} provides $\kappa_0$ such that 
         \[
         \kappa_0 \ell(\c,\c^*) \geq \|\c - \c^*(\c)\|^2.
         \]     
         Choosing $K_2= 8 R(\c^*)\kappa_0$ in Proposition \ref{concentrationpremierterme}, $K_3= 2 C_\infty^2 \kappa_0$
in Proposition \ref{concentrationsecondterme} and summing the two resulting inequalities leads to  \eqref{secondinequality}, valid on a set which has probability larger than $1-2e^{-x}$.

 \subsection{Proof of Proposition \ref{minimax}}\label{Proofofminimax}
 Throughout this subsection, for a codebook $\c$, let $Q$ denote  the associated nearest-neighbor quantizer. In the general case, such an association depends on how the boundaries are allocated. However, since the distributions involved in the minimax result have a density, how boundaries are allocated will not matter. 
         
          Let $k \geq 3$ be an integer. For convenience  $k$ is assumed to be divisible by $3$. Let $m=2k/3 $. Let $z_1, \hdots, z_m$ denote a $6 \Delta$-net in $\mathcal{B}(0,M-\rho)$, where $\Delta>0$, and $w_1, \hdots, w_m$ a sequence of vectors such that $\|w_i\| = \Delta$. Finally, denote by $U_i$ the ball $\mathcal{B}(z_i,\rho)$ and by $U'_i$ the ball $\mathcal{B}(z_i,\rho)$. Slightly anticipating, define $\rho = \frac{\Delta}{16}$.
         
          To get the largest $\Delta$ such that for all $i=1,\hdots,k$ $U_i$ and $U'_i$ are included in $\mathcal{B}(0,M)$, it suffices to get the largest $\Delta$ such that there exists a $6 \Delta$-net in $\mathcal{B}(0,M-\Delta/16)$. Since the cardinal of a $6 \Delta$-net is larger than the largest number of balls of radius $6\Delta$ which can be packed into $\mathcal{B}(0,M-\Delta/16)$, a sufficient condition on $\Delta$ to guarantee that a $6\Delta$-net can be found is given by
\[
m \leq \left (\frac{M-\Delta/16}{6\Delta} \right )^d.
\]
Since $\Delta \leq M$, $\Delta$ can be chosen as 
\[
\Delta = \frac{15M}{96 m^{1/d}}.
\]
For such a $\Delta$, $\rho$ takes the value $\rho = \frac{\Delta}{16} = \frac{15M}{1536 m^{1/d}}$. Therefore, it only depends on $k$, $d$, and $M$.

          Let $z=(z_i)_{i=1,\hdots,m}$ and $w=(w_i)_{i=1,\hdots,m}$ be sequences as described above, such that, for $i=1,\hdots,k$, $U_i$ and $U'_i$  are included in $\mathcal{B}(0,M)$.
For a fixed $\sigma \in \left\{-1,+1 \right \}^{m}$ such that $\sum_{i=1}^{m}{\sigma_i}=0$, let $P_\sigma$ be  defined  as
\[
\left \{\mbox{
\begin{tabular}{@{}ccc}
$P_{\sigma}(U_i)$&$=$& $\frac{1+\sigma_i\delta}{2m}$ \\ 
$P_{\sigma}(U'_i)$&$=$&$\frac{1+\sigma_i\delta}{2m}$ \\
$P_{\sigma}$&$ \underset{U_i}{\sim}$ &$  ( \rho - \| x - z_i \|) \mathbbm{1}_{\| x - z_i \| \leq \rho} d\lambda(x)$\\
$P_{\sigma}$& $\underset{U'_i}{\sim}$ & $ ( \rho - \| x - z_i - w_i \|)\mathbbm{1}_{\| x - z_i - w_i \| \leq \rho} d\lambda(x)$,
\end{tabular}
}
\right . 
\]
 where $\lambda$ denote the Lebesgue measure. These cone-shaped distributions are designed to have a continuous density, as done in \cite{Antos05}. To be more precise, for $\tau$ in $\{-1,+1\}^{\frac{m}{2}}$, $\sigma(\tau)$ is defined as the sequence in $\{-1,+1\}^{m}$ such that
\[
\left \{
\begin{array}{@{}ccc}
\sigma_i(\tau)&=&\tau_i\\
\sigma_{i+ \frac{m}{2}}(\tau)&=&-\sigma_i(\tau),
\end{array}
\right.
\]
for $i=1,\hdots,\frac{m}{2}$. Finally, for a quantizer $Q$ let $R(Q,P_\sigma)$ denote the distortion of $Q$ in the case where the source distribution is $P_\sigma$. 
      
      Similarly, for $\sigma$ in $\{-1,+1\}^{m}$ satisfying $\sum_{i=1}^{m}{\sigma_i} =0$, let $Q_\sigma$ denote the quantizer defined by $Q_\sigma(U_i) = Q_\sigma(U'_i) = z_i + \omega_i/2$ if $\sigma_i=-1$, $Q_\sigma(U_i) = z_i$ and $Q_\sigma(U'_i) = z_i + \omega_i$ if $\sigma_i = +1$. Let $\mathcal{Q}$ denote the set of such quantizers. It can be proved that only quantizers in $\mathcal{Q}$ have to be considered. 
      
      \begin{prop}\label{definitionDelta}
      Assume that $A \geq 6$, $\delta \leq 1/3$, $\Delta > 0$, and $\rho \leq \frac{\Delta}{16}$.  Then, for every quantizer $Q$ there exists a quantizer $Q_\sigma$ in $\mathcal{Q}$ such that 
\[
\forall P_{\sigma'}  \quad R(Q_\sigma,P_{\sigma'}) \leq R(Q,P_{\sigma'}).
\]
\end{prop}

The proof of Proposition \ref{definitionDelta} follows the proof of Step 3 of Theorem 1 in \cite{Bartlett98}, replacing distributions supported on a finite set with distributions supported on small balls. Provided that the radius of these balls are small enough, the results are nearly the same in the two cases. The proof of Proposition \ref{definitionDelta} is given in Section \ref{ProofofPropositiondefinitionDelta}.

 For any $\sigma$ and $\sigma'$  in $\left \{-1,1 \right \}^{{m}}$, denote by $\rho(\sigma,\sigma') = \sum_{i=1}^{{m}}{|\sigma_i - \sigma'_i|}$, and by $H(P_\sigma, P_{\sigma'})$ the Hellinger distance between $P_\sigma$ and $P_{\sigma'}$. To apply Assouad's Lemma to the set $\{P_{\sigma(\tau)}\}_{\tau \in \{-1,+1 \}^{\frac{m}{2}}}$, the following lemma is needed:

\begin{lem}\label{Assouadtechnique}
Let $\tau$ and $\tau'$ denote two sequences in $\{-1,+1\}^{\frac{m}{2}}$ such that $\rho(\tau,\tau') =2$, then
\[
H(P^{\otimes n}_{\sigma(\tau)}, P^{\otimes n}_{\sigma(\tau')}) \leq \frac{4n\delta^2}{m},
\]
where $P^{\otimes n}$ denotes the product law of a $n$-sample drawn from $P$.

Furthermore, for any $\sigma$ and $\sigma'$ in $\{-1,+1\}^{m}$,
\[
R(Q_{\sigma'}, P_\sigma) = R(Q_\sigma, P_\sigma) + \frac{\Delta^2 \delta}{8m} \rho(\sigma,\sigma').
\]
\end{lem}

Equipped with Lemma \ref{Assouadtechnique}, a direct application of Assouad's Lemma as in Theorem 2.12 of \cite{Tsybakov09} yields that, provided that $\delta = \frac{\sqrt{m}}{2\sqrt{n}}$,
\[
\mathbb{E}\sup_{\tau \in \{-1,+1\}^{\frac{m}{2}}}{ R(\hat{Q}_n,P_{\sigma(\tau)}) - R(Q_{\sigma(\tau)},P_{\sigma(\tau)}))} \geq c_0 M^2 \sqrt{\frac{k^{1-\frac{4}{d}}}{n}},
\]
for any empirically designed quantizer $\hat{Q}_n$, where $c_0$ is an explicit constant.

Finally, it may be noticed that, for every $\delta \leq \frac{1}{3}$ and $\sigma$, $P_\sigma$ satisfies a margin condition as in \eqref{conditionmargeminimax}, and is $\varepsilon$-separated, with
\[
\varepsilon = \frac{\Delta^2 \delta}{2m}.
\]
This concludes the proof of Proposition \ref{minimax}.

As a remark, it is worth mentioning that, whenever $\varepsilon \sim 1/n^{\alpha}$, with $\alpha < 1/2$, no interesting minimax lower bound can be derived using the distributions $\{P_\sigma\}$'s. In fact, it can be proved, making use of the approach exposed in \cite{Bartlett98}, that the empirical risk minimization strategy $\hat{Q}_n$ achieves the uniform rate 
 \[
       \sup_{\{P_\sigma\}_{\sigma \in \{-1,+1\}^{m}}}{R(\hat{Q}_n,P_\sigma) - R(Q_\sigma,P_\sigma)} \leq C(M,k,d)n^{-\alpha} e^{-c(k)n^{1-2 \alpha}},
       \]
			where $C(M,k,d)$ and $c(k)$ are constants. Consequently, in order to get a minimax lower bound matching the bound offered in Theorem \ref{mainresult}, more general probability distributions should be used.

 \subsection{Proof of Proposition \ref{Gaussianboundary}}
 
 As mentioned below Proposition \ref{Gaussianboundary}, the inequality 
 \[
                  \frac{\theta_{min}}{\theta_{max}} \geq \frac{2048 k \sigma^2}{(1-\varepsilon) \tilde{B}^2(1 - e^{-\tilde{B}^2/{2048\sigma^2}})},
                  \]  
                  ensures that, for every $j$ in $\left \{ 1,\hdots, k \right \}$, there exists $i$ in $\left \{ 1,\hdots, k \right \}$ such that $\left \|c_i^* - m_j \right \| \leq \tilde{B}/16$. To be more precise, let $ \mathbf{m}$  denote the vector of means  $(m_1, \hdots, m_k)$, then
                  \begin{align*}
                  R(\mathbf{m}) & \leq \sum_{i=1}^{k}{\frac{\theta_i}{2 \pi \sigma^2 N_i} \int_{V_i(\mathbf{m})}{\|x - m_i\|^2 e^{-\frac{\| x - m_i \|^2}{2 \sigma^2}}dx}} \\
                                & \leq \frac{p_{max}}{2(1-\varepsilon) \pi \sigma^2} \sum_{i=1}^{k}{ \int_{\mathbb{R}^2}{\|x - m_i\|^2 e^{-\frac{\| x - m_i \|^2}{2 \sigma^2}}dx}} \\
                                & \leq \frac{2kp_{max} \sigma^2}{1-\varepsilon}.
                    \end{align*}  
                    Assume that there exists $i$ in $\left \{1, \hdots, k \right \}$ such that, for all $j$, $\| c_j^* - m_i \| \geq \tilde{B}/16$. Then
                    \begin{align*}
                    R(\c) & \geq \frac{\theta_i}{2 \pi \sigma^2}\int_{\mathcal{B}(m_i,\tilde{B}/32)}{\frac{\tilde{B}^2}{1024} e^{-\frac{\|x - m_i\|^2}{2 \sigma^2}}} \\
                          & \geq \frac{\tilde{B}^2 \theta_{min}}{2048 \pi \sigma^2} \int_{\mathcal{B}(m_i,\tilde{B}/32)}{e^{-\frac{\|x - m_i\|^2}{2 \sigma^2}}} \\
                          & > \frac{\tilde{B}^2 \theta_{min}}{1024} \left ( 1 - e^{-\frac{\tilde{B}^2}{2048 \sigma^2}} \right ) \\
                          & > R(\mathbf{m}).
                    \end{align*}           
                     Hence the contradiction. Up to relabeling, it is now assumed that for $i=1,\hdots,k$, $\|m_i - c_i^*\| \leq \tilde{B}/16$. Take $y$ in $N^*(x)$, for $x \leq \frac{\tilde{B}}{8}$, then, for every $i$ in $\left \{ 1, \hdots, k \right \}$, 
                     \[
                     \|y - m_i \| \geq \frac{\tilde{B}}{4},
                     \]
                     which leads to
                     \begin{align*}
                     \sum_{i=1}^{k}{\frac{\theta_i}{2 \pi \sigma^2 N_i}{\|y - m_i\|^2 e^{-\frac{\| y - m_i \|^2}{2 \sigma^2}}}} \leq \frac{k \theta_{max}}{(1-\varepsilon) 2 \pi \sigma^2} e^{-\frac{\tilde{B}^2}{32 \sigma^2}}.
                     \end{align*}
 Since the Lebesgue measure of $N^*(x)$ is smaller than $4k \pi M x$, it follows that
                     \[
                     P(N^*(x)) \leq \frac{2 k^2 M \theta_{max}}{(1-\varepsilon) \sigma^2} e^{-\frac{\tilde{B}^2}{32 \sigma^2}}x . 
                     \]
                     On the other hand, $\|m_i - c_i^*\| \leq \tilde{B}/16$ yields that
                     \[
                     \mathcal{B}(m_i,3 \tilde{B} /8) \subset V_i(\c^*).
                     \]
                     Therefore, 
                     \begin{align*}
                     P(V_i(\c^*)) & \geq \frac{\theta_i}{2 \pi \sigma^2 N_i} \int_{\mathcal{B}(m_i,3 \tilde{B} /8)}{e^{-\frac{\|x - m_i \|^2}{2 \sigma^2}}dx} \\
                     & \geq \theta_i \left (1 - e^{-\frac{9 \tilde{B}^2}{128 \sigma^2}} \right ),
                     \end{align*}
                     hence $p_{min} \geq \theta_{min}\left (1 - e^{-\frac{9 \tilde{B}^2}{128 \sigma^2}} \right )$. Consequently, provided that
                     \[
                     \frac{\theta_{min}}{\theta_{max}} \geq  \frac{2048 k^2 M^3}{(1 - \varepsilon)7\sigma^2 \tilde{B}(e^{\tilde{B}^2/{32\sigma^2}}-1)},
                     \]
                     direct calculation shows that
                     \[
                     P(N^*(x)) \leq \frac{B p_{min}}{128 M^2}x.
                     \]

\section{Technical results}\label{Technical results}

\subsection{Proof of Proposition \ref{chainagechichignoud}}\label{ProofofPropositionchainagechichignoud}

The proof of Proposition \ref{chainagechichignoud} is derived from the proof of Lemma 3 in \cite{Chichi13}. Let $\c^*$ be an optimal codebook, and $\c$ be a codebook. We denote by $f_{\c^*,\c}$ the function $\gamma(\c,.) - \gamma(\c^*,.)$, so that
\[
\mathcal{F}_1= \left \{ f_{\c^*(\c),\c} | \quad \c \in \mathcal{B}(0,M)^k \right \}.
\]
Let $\Psi_1(r)$ denote the function
\[
\Psi_1(r) = \mathbb{E} \sup_{\c^* \in \mathcal{M}, \| \c - \c^*(\c) \| \leq r}{\left |(P-P_n)f_{\c^*(\c),\c}\right |}.
\]
Since 
\[
\left \{ (\c^*(\c),\c^*)| \quad \| \c - \c^*(\c) \| \leq r \right \} \subset \left \{ (\c^*,\c) | \quad \c^* \in \mathcal{M}, \| \c - \c^* \| \leq r\right \}, 
\]
it is easy to see that
\[
\Psi_1(r) \leq \mathbb{E} \sup_{\c^* \in \mathcal{M}, \| \c - \c^* \| \leq r}{\left |(P-P_n)f_{\c^*,\c}\right |}.
\]

The rest of the proof is derived from a chaining technique, used in the proof of Proposition 5.1 in \cite{Levrard12} or Lemma 3 in \cite{Chichi13}. 
Set $\varepsilon_j = 2^{-j}r$, for $j \geq 0$, and for every $\c^*$ in $\mathcal{M}$, denote by $N_j(\c^*)$ an $\varepsilon_j$ net of $\mathcal{B}(\c^*,r)$, such that for every $\c$ in $\mathcal{B}(\c^*,r)$ there exists $\c_j$ in $N_j(\c^*)$ such that $\| \c_j - \c \| \leq \varepsilon_j$. According to the proof of Theorem 2 in \cite{Antos04} or Lemma 3 in \cite{Chichi13}, such an $N_j(\c^*)$ can be defined, with
\[
|N_j(\c^*)| \leq \left( \frac{2r \sqrt{kd}}{\varepsilon_j} \right ):=n(\varepsilon_j).
\]
By a dominated convergence Theorem, for any fixed $\c^*$ in $\mathcal{M}$ and $\c$, 
\[
f_{\c^*,\c_j} \underset{j \rightarrow \infty }{\overset{L_1(P), a.s.}{\longrightarrow}} f_{\c^*,\c}.
\]
This allows us to decompose the expression of $\Psi_1$ as follows.
\begin{align*}
\Psi_1(r) & \leq \mathbb{E} \sup_{\c^* \in \mathcal{M}, \| \c - \c^* \| \leq r}{\left |(P-P_n)f_{\c^*,\c}\right |} \\
& \begin{multlined} \leq \mathbb{E} \sup_{\c^* \in \mathcal{M}, \c_0 \in N_0(\c^*)}{\left |(P-P_n)f_{\c^*,\c_0}\right |} \\ + \sum_{j >1}{\mathbb{E} \sup_{\c^* \in \mathcal{M}, \c_j \in N_j(\c^*), \c_{j-1} \in N_{j-1}(\c^*)}{\left |(P-P_n)(f_{\c^*,\c_j}-f_{\c^*,\c_{j-1}})\right |}}, \end{multlined} \\
 & := A_1 +A_2.
\end{align*}
It remains to bound from above these two terms.

\textbf{Bound on $A_1$}

Introducing some Rademacher random variables $\sigma_i$, $i=1, \hdots,n$ and using the symmetrization principle as in \cite{Koltchinskii04} leads to

\begin{multline*}
\mathbb{E} \sup_{\c^* \in \mathcal{M}, \c_0 \in N_0(\c^*)}{\left |(P-P_n)f_{\c^*,\c_0}\right |} \\
\leq 2 \mathbb{E}_X \mathbb{E}_\sigma \sup_{\lambda = \pm 1, \c^* \in \mathcal{M}, \c_0 \in N_0(\c^*)}{\frac{1}{n} \sum_{i=1}^{n}{\sigma_i \lambda f_{\c^*,\c_0}(X_i)}}.
\end{multline*}
Let introduce here a maximal inequality derived from Lemma 2.3 in \cite{Massart03}.
\begin{lem}\label{maximalinequality}
Let $x_1, \hdots, x_n$  denote a sequence of points in $\mathcal{X}$, and let $\sigma_1, \hdots, \sigma_n$ denote a sequence of independent Rademacher random variables. Let $\mathcal{F}$ be a set of real valued functions over $\mathcal{X}$ such that $|\mathcal{F}| < \infty$, and 
\[
\sup_{f \in \mathcal{F}}{\frac{1}{n} \sum_{i=1}^{n}{f^2(x_i)}} \leq v.
\]
Then
\[
\mathbb{E}_{\sigma} \sup_{f \in \mathcal{F}}{\frac{1}{n}\sum_{i=1}^{n}{\sigma_i f(x_i)}} \leq \sqrt{2 v \log(|\mathcal{F}|)}.
\]
\end{lem}
In our case, for all $\c^*$ in $\mathcal{M}$ and $\c_0$ in $N_0(\c^*)$,
\[
\frac{1}{n} \sum_{i=1}^{n}{f^2_{\c^*,\c_0}(X_i)} \leq \frac{16 M^2 r^2}{n},
\]
and 
\[
\left | \left \{ \lambda = \pm 1, \c^* \in \mathcal{M}, \c_0 \in N_0(\c^*) \right \} \right | \leq | \mathcal{M} | (4\sqrt{kd})^{kd}.
\]
Therefore, a direct application of Lemma \ref{maximalinequality} yields that
\[
A_1 \leq \frac{8 \sqrt{2} M}{\sqrt{n}}\sqrt{kd\log(4 |\mathcal{M}| \sqrt{kd})}.
\]

\textbf{Bound on $A_2$}

Let $j>1$. Using the same symmetrization argument as above leads to
\begin{align*}
A_{2,j} & :=\mathbb{E} \sup_{\c^* \in \mathcal{M}, \c_j \in N_j(\c^*), \c_{j-1} \in N_{j-1}(\c^*)}{\left |(P-P_n)(f_{\c^*,\c_j}-f_{\c^*,\c_{j-1}})\right |} \\
& \leq \mathbb{E}_X\mathbb{E}_\sigma \sup_{\scriptsize \begin{aligned} \lambda = \pm 1,& \c^* \in \mathcal{M}, \\ \c_j \in N_j(\c^*),& \c_{j-1} \in N_{j-1}(\c^*) \end{aligned}}{\frac{1}{n}\sum_{i=1}^{n}{\sigma_i \lambda (f_{\c^*,\c_j}(X_i)-f_{\c^*,\c_{j-1}}(X_i))}}.
\end{align*}

Since
\[
\left \|f_{\c^*,\c_j}(.)-f_{\c^*,\c_{j-1}}(.)\right\|_{\infty} \leq 8Mr 2^{-(j-1)},
\]
and
\[
\left | \left\{ \lambda = \pm 1, \c^* \in \mathcal{M}, \c_j \in N_j(\c^*), \c_{j-1} \in N_{j-1}(\c^*) \right \} \right | \leq 2 |\mathcal{M}|n(\varepsilon_j)^2,
\]
a direct application of Lemma \ref{maximalinequality} leads to 
\[
A_{2,j} \leq 64 M r \sqrt{kd\log(|\mathcal{M}| \sqrt{kd} 2^{j+2})} 2^{-(j-1)}.
\]
Comparing a sum with an integral, and observing that 
\[
\int_{0}^{1} {\sqrt{\log(-x)}dx} \leq 1
\] 
ensures that
\[
A_2 = \sum_{j >1} A_{2,j} \leq \frac{256 Mr}{\sqrt{n}} \left ( \sqrt{\log(|\mathcal{M}| \sqrt{kd})} +1 \right ).
\]
Combining the two bounds and remarking that 
\[
\mathbb{E} \sup_{f \in \mathcal{F}_1, \omega_1(f) \leq \delta}{\left | (P - P_n)f \right |} \leq \Psi_1\left ( \frac{\sqrt{\delta}}{4M} \right )
\]
gives the result of Proposition \ref{chainagechichignoud}.

\subsection{Proof of Proposition \ref{complexitesecondterme}}\label{complexitebourrin}

         The proof of Proposition \ref{complexitesecondterme} is based on a sharper chaining technique than the one used in Proposition 5.1 in \cite{Levrard12}. We intend to bound from above the complexity term 
         \[
         \mathbb{E} \sup_{\omega_3(f) \leq \delta, f \in \mathcal{F}_3} |(P - P_n)f|.
         \]
         To this aim, define
         \[
         \Psi_3(r) = \mathbb{E} \sup_{  \| \c - \c^*(\c) \| \leq r}{\left |(P-P_n)R(\c,\c^*,.) \right |},
         \] 
        where we recall that 
        \begin{multline*}
        R(\c,\c^*,x) = \sum_{i,j =1, \hdots, k} {\mathbbm{1}_{V_i(\c^*)\cap W_j(\c)} \| \c - \c^* \|^{-1} \left [ \vphantom{\left\langle x - \frac{c_i + c_j}{2}, c_i - c_j\right\rangle} \| c_i - c_i^*\|^2 \right. } \\ \left .+ 2 \left\langle x - \frac{c_i + c_j}{2}, c_i - c_j\right\rangle \right ],
        \end{multline*}
        where $(W_1(\c), \hdots, W_k(\c))$ is a Voronoi partition, defined in Section \ref{Notation}. For technical reasons, this Voronoi partition must be specified. Denote by $\mathcal{C}(p)$ the set of subsets of $\mathbb{R}^d$ made of intersections of at most $p$ half spaces (closed or open).
        
        Since $R(\c,\c^*,.)$ does not depend on how ties are broken, or, in other words, $R(\c,\c^*,.)$ does not depend on the choice of $W_j(\c)$'s among the partition cells satisfying
       \[
       \overset{o}{V}_j(\c) \subset W_j(\c) \subset V_j(\c),
       \] 
       we choose a Voronoi partition such that every $W_j$ $\in$ $\mathcal{C}(k-1)$. For instance, if $H_{i,j}$ denotes the closed half-space $\{ \|x-c_i\| \leq \|x-c_j\| \}$ and $\overset{o}{H}_{i,j}$ the open half space $\{ \| x - c_i \| < \|x-c_j\| \}$. It is possible to build a Voronoi partition such that every cell is in $\mathcal{C}(k-1)$, choosing 
       \[
       W_j(\c) = \bigcap_{i<j}{\overset{o}{H}_{i,j}} \cap \bigcap_{i>j}{H_{i,j}}.
       \] 
       In short, this convention consists in allocating points on boundaries between $V_j$'s to the smallest possible index. As a consequence, it is immediate that
       \[
       \Psi_3(r) = \mathbb{E} \sup_{ \| \c - \c^*(\c) \| \leq r, W_j(\c) \in \mathcal{C}(k-1)}{\left |(P-P_n)R(\c,\c^*,.) \right |}.
       \]
The following set of function of interest  is then introduced.
          \begin{multline*}
 \mathcal{G}(r) = \{0\} \cup \frac{1}{F_r}\left \{  \lambda \| \c - \c^*\|^{-1} \sum_{i,j}{\mathbbm{1}_{V_i(\c^*) \cap W_j(\c)\cap\mathcal{B}(0,M)} \left ( \vphantom{\left\langle x - \frac{c_i + c_j}{2}, c_i - c_j\right\rangle} \| c_i - c_i^* \|^2 \right. } \right.\\ 
        + \left. { \left. 2 \left\langle x - \frac{c_i + c_j}{2}, c_i - c_j\right\rangle \right )} | \quad \lambda = \pm1, \c \in \mathcal{B}(0,M)^k, \c^* \in \mathcal{M}, W_j \in \mathcal{C}(k-1) \vphantom{\sum_{i,j}} \right \},
        \end{multline*}
 where $F_r$ is defined in \eqref{definitionF} as an envelope of $\mathcal{G}(r)$.
 
        Let $\sigma_1, \hdots, \sigma_n$ denote a sequence of independent Rademacher variables. As developed in the proof of Proposition \ref{chainagechichignoud}, the first step is a symmetrization inequality
        \begin{align*}
        \Psi_3(r) & \leq \mathbb{E} \sup_{\c^* \in \mathcal{M},  \| \c - \c^* \| \leq r}{\left |(P-P_n)R(\c,\c^*,.) \right |}\\
                  & \leq 2 \mathbb{E}_X \mathbb{E}_\sigma \sup_{\lambda = \pm 1, \c^* \in \mathcal{M},  \| \c - \c^* \| \leq r}{\frac{1}{n} \sum_{i=1}^{n} {\lambda \sigma_i R(\c,\c^*,X_i)}}\\
                  &\leq 2 \mathbb{E}_X \mathbb{E}_\sigma \sup_{g \in \mathcal{G}(r)}{\frac{1}{n} \sum_{i=1}^{n} {\sigma_i F_r(X_i) g(X_i)}}        \\         
                  & := 2 \mathbb{E}_X \mathcal{R}_n.
         \end{align*}

         The next step is to chain the set $\mathcal{G}(r)$. To this aim, define for any set of real valued function $\mathcal{F}$, any norm $\| \quad \|$ on $\mathcal{F}$ and any $\varepsilon > 0$, the covering number $\mathcal{N}(\mathcal{F}, \| \quad \| , \varepsilon)$ as the cardinal of the smallest covering of $\mathcal{F}$ with balls of radius $\varepsilon$ for the norm $\| \quad \| $. 
         
         To be more precise, for any $g$ in $\mathcal{G}(r)$, and any finite subset $S \subset \mathbb{R^d}$, we define
         \[
         \|g\|_{L_2(S)} = \sqrt{\frac{1}{|S|}\sum_{s \in S}{g^2(s)}}, 
         \]
         and, with a slight abuse of notation, $\|g\|_{L_2(P_n)} = \sqrt{1/n\sum_{i=1}^{n}{g^2(X_i)}}$. The technical result concerning the covering numbers of $\mathcal{G}(r)$ is the following.
          \begin{prop}\label{chaining}
 Let $S$ be a finite set, and $0< \varepsilon <1$. There exists some constant $K >0$, not depending on $S$, such that
 \[
 \mathcal{N}(\mathcal{G}(r),\varepsilon,L_2(S))) \leq \left (\frac{k^2(4k-2)}{\varepsilon} \right )^{K P(k,d)},
 \]
 with $P(k,d)=k^2(2(k-1)(d+1) +24(3d+4))$.
 \end{prop}
 
 For clarity, the proof of Proposition \ref{chaining} is postponed to the following subsection. An immediate consequence of Proposition \ref{chaining} is that
 \[
 \mathcal{N}(\mathcal{G}(r),\varepsilon,L_2(P_n)) \leq \left (\frac{k^2(4k-2)}{\varepsilon} \right )^{K P(k,d)} := n(\varepsilon),
 \]
 for any $n$-sample $X_1, \hdots, X_n$. Consequently, let $X_1, \hdots, X_n$ be fixed, and set $\varepsilon_0=1$, $\varepsilon_j = 2^{-2j}\varepsilon_0$, for $j>1$. 
 
  For $j=0$, since $F_r$ is an envelope of $\mathcal{G}(r)$, a $1$ covering of $\mathcal{G}(r)$ for the $L_2(P_n)$ norm is the ball of center $g_0=0$ and radius $1$.
  
  For $j>1$, Proposition \ref{chaining} provides a $\varepsilon_j$ covering $\mathcal{G}_j(r)$ of $\mathcal{G}(r)$ for the $L_2(P_n)$ norm with cardinality at most $n(\varepsilon_j)$. For any $g$ in $\mathcal{G}(r)$, denote by $g_j$ the projection of $g$ onto this covering, so that $\| g - g_j \|_{L_2(P_n)} \leq \varepsilon_j$. For short we will write $n_j = n(\varepsilon_j)$.
  
  It is easy to see that for every $i$ in $\left \{1, \hdots, n \right \}$, 
  \[
  g_j(X_i) \underset{j \rightarrow \infty}{\longrightarrow} g(X_i).
  \]  
  Then $\mathcal{R}_n$ may be decomposed as follows:
  \begin{align*}
  \mathcal{R}_n & = \mathbb{E}_\sigma \sup_{\lambda = \pm 1, \c^* \in \mathcal{M},  \| \c - \c^* \| \leq r}{\frac{1}{n} \sum_{i=1}^{n} {\lambda \sigma_i R(\c,\c^*,X_i)}} \\
  & \leq \sum_{j>1}{\mathbb{E}_\sigma \sup_{g \in \mathcal{G}}{ \frac{1}{n}\sum_{i=1}^{n}{\sigma_i F_r(X_i)(g_j(X_i) - g_{j-1}(X_i))}}} \\
  & := \sum_{j>1}{b_j}.
  \end{align*} 
  A direct application of Lemma \ref{maximalinequality} for every $b_j$ yields that
  \begin{align*}
  b_j & \leq \frac{1}{\sqrt{n}}\sqrt{2 \sup_{g \in \mathcal{G}(r)}{\left \| (g_j -g_{j-1})F_r \right \|^2_{L_2(P_n)} \log(n_j n_{j-1})}} \\
  & \leq \frac{1}{\sqrt{n}}\sqrt{2 \log(n_j n_{j-1}) \sup_{g \in \mathcal{G}(r)}{2 C_\infty \|F_r\|_{L_2(P_n)}\|g_j-g_{j-1}\|_{L_2(P_n)}}} \\
  & \leq \frac{4 \sqrt{ C_{\infty}}}{\sqrt{n}}\sqrt{\|F_r\|_{L_2(P_n)}\|}\sqrt{\log(n(\varepsilon_j))}\sqrt{\varepsilon_{j-1}}.
  \end{align*} 
        Denote by $\varepsilon'_j$ the quantity $ \sqrt{\varepsilon_j} = 2^{-j}$. Since $x \mapsto \sqrt{\log(n(x^2))}$ is non-increasing, it is quite easy to see that
        \[
        \frac{\sqrt{\log(n(\varepsilon_j^{'2}))} \varepsilon'_{j-1}}{4} = \sqrt{\log(n(\varepsilon_j^{'2}))}\varepsilon'_{j+1} \leq \int_{\varepsilon'_{j+1}}^{\varepsilon'_j}{\sqrt{\log(n(x^2))}dx}.
        \]
        From this we deduce that
        \[
        \begin{aligned}
        \sum_{j >1}\sqrt{\varepsilon_{j-1}\log(n(\varepsilon_j))} &\leq 4 \int_{0}^{1/2}{\sqrt{\log(n(x^2))}dx} \\
        & \leq \int_{0}^{1/2}{\sqrt{KP(k,d)\log(\frac{k^2(4k-2)}{x^2})}dx} \\
        & \begin{multlined}\leq 2 \sqrt{KP(k,d) \log(k^2(4k-2))} \\ + 4 \sqrt{KP(k,d)} \int_{0}^{1/2}{\sqrt{\log(1/x^2)}}.\end{multlined}
        \end{aligned}
        \]
        Since $\int_{0}^{1/2}{\sqrt{\log(1/x^2)}} \leq 1$, we get
        \[
        \sum_{j >1}\sqrt{\varepsilon_{j-1}\log(n(\varepsilon_j))} \leq 8 \sqrt{KP(k,d)\log(k^2(4k-2))} :=Q(k,d).
        \]  
        Thus
        \[
        \mathcal{R}_n \leq \frac{4 Q(k,d)\sqrt{ C_{\infty}}}{\sqrt{n}}\sqrt{\|F_r\|_{L_2(P_n)}\|}.
        \]
        It remains now to take expectations with respect to the $n$-sample $X_1, \hdots, X_n$.
        Since $x \mapsto \sqrt{x}$ is a concave map,
        \[
        \mathbb{E}_X(\sqrt{\|F_r\|_{L_2(P_n)}}) \leq \sqrt{\|F_r\|_{L_2(P)}} = \sqrt{C_2(r)}.
        \]
        Gathering all terms leads to 
        \[
        \Psi_3(r) \leq \frac{8Q(k,d)\sqrt{ C_{\infty} C_2(r)}}{\sqrt{n}}.
        \]        
        Substituting $\sqrt{\delta}/C_\infty$ with $r$ gives the result of Proposition \ref{complexitesecondterme}.

 \subsection{Proof of Proposition \ref{chaining}}\label{ProofofPropositionchaining}
 
        Let $S$ be a finite subset of $\mathbb{R}^d$, and denote by $\mathcal{C}(p)$ the set of subsets of $\mathbb{R}^d$ which are intersections of at most $p$ half spaces (closed or open). For short, $\mathcal{N}(\mathcal{F},\varepsilon)$ will denote $\mathcal{N}(\mathcal{F},L_2(S),\varepsilon)$.
        
                The proof of Proposition \ref{chaining} is based on the following result, Theorem 1 in \cite{Mendelson03}.
        \begin{thm}\label{Mendelson}
 Let $P$ denote a measure on $\Omega$. Let $\mathcal{F}$ be a set of maps from $\Omega$ into $[-1,1]$. Then, for every $0<t<1$,
 \[
 \mathcal{N}(\mathcal{F},t,L_2(P)) \leq \left ( \frac{2}{t} \right )^{Kvc(\mathcal{F},ct)},
 \]
 where $K$ and $c$ are constants, and $vc(\mathcal{F},ct)$ denotes the $t$-shattering dimension of $\mathcal{F}$, as defined in \cite{Mendelson03}. 
 \end{thm}
 Remark that, for every $t>0$,  $vc(\mathcal{F},ct) \leq d_p(\mathcal{F})$, where $d_p(\mathcal{F})$ denotes the pseudo-dimension of $\mathcal{F}$, that is the largest integer $p$ such that there exists $x_1, \hdots, x_p \in \Omega$, and $t_1, \hdots, t_p$ real numbers, satisfying the following property: for every $\sigma \in \{-1,1\}^p$ there exists $f_\sigma \in \mathcal{F}$ such that, for $i=1,\hdots,p$, $\sigma_i (f_\sigma(x_i) - t_i) >0$. As a consequence, the quantity of interest is $d_{p}(\mathcal{G}(r))$.
 
    Recalling that every $g$ in $\mathcal{G}$ can be written
    \begin{multline*}
    R(\c,\c^*,x)  = \|\c-\c^*\|^{-1}  \sum_{i,j}{\mathbbm{1}_{V_i(\c^*) \cap W_j(\c)\cap \mathcal{B}(0,M)}\left( \vphantom{\left\langle x - \frac{c_i + c_j}{2}, c_i - c_j\right\rangle} \| c_i - c_i^*\|^2 \right .} \\
    {\left . + 2 \left\langle x - \frac{c_i + c_j}{2}, c_i - c_j\right\rangle \right )},
    \end{multline*}
     where, for every $j$ in $\left \{1,\hdots,k\right \}$, $W_j(\c)$ is in $\mathcal{C}(k-1)$, it may be noticed that $g$ takes the form of a sum of $k^2$ maps of the type $\ell \mathbbm{1}_C \mathbbm{1}_{\mathcal{B}(0,1)}$, where $\ell$ denotes an affine map, and $C$ is an element of $\mathcal{C}(2(k-1))$. Let $\mathcal{A}ff(\mathbb{R}^d,\mathbb{R})$ denote the space of affine maps between $\mathbb{R}^d$ and $\mathbb{R}$.
     
      It is worth pointing out that every map $\ell \mathbbm{1}_C \mathbbm{1}_{\mathcal{B}(0,1)}$ involved in the above decomposition of $R(\c,\c^*,.)$ admits $F_r$ as an envelop.
 
 Denote by 
 \[
 \mathcal{G}'(r) = \left \{ \frac{\ell \mathbbm{1}_C \mathbbm{1}_\mathcal{B}(0,M)}{F_r}; \ell \in \mathcal{A}ff(\mathbb{R}^d,\mathbb{R}), C \in \mathcal{C}(2(k-1))  \right \}.
 \]
 We immediately deduce that
 \[
 \mathcal{N}(\mathcal{G}(r),\varepsilon) \leq \left ( \mathcal{N}(\mathcal{G}'(r),\varepsilon/k^2) \right )^{k^2}.
 \]
 Consider now the set of functions $\mathcal{N}(\mathcal{G}'(r),\varepsilon)$. The following lemma offers a decomposition of $\mathcal{N}(\mathcal{G'}(r),\varepsilon)$. 
 \begin{lem}\label{multiplication}
 Denote by $\mathcal{F}_s$ $s=1,\hdots,p$ a collection of set of functions taking values in $\left [-1,1 \right ]$. Then
 \[
 \mathcal{N}\left (\prod_{s=1}^{p}{\mathcal{F}_s},\varepsilon \right ) \leq \prod_{s=1}^{p}{\mathcal{N}(\mathcal{F}_s,\varepsilon/p)}.
 \]
 \end{lem}
 In order to apply Lemma \ref{multiplication}, a crucial point is to only deal with maps taking values in $[-1,1]$. To this aim, we define the set 
 \[
 \mathcal{H} = \left \{ \frac{f\mathbbm{1}_{\{|f | \leq F_r\}}}{F_r} \right \},
 \]
  where $f$ is in $\mathcal{A}ff(\mathbb{R}^d,\mathbb{R})$, and the set
  \[
  \mathcal{P}_d =  \left \{ {\left \{\mathbbm{1}_{\{\ell \leq a\}}|  \ell \in \mathcal{L}(\mathbb{R}^d,\mathbb{R}), a \in \mathbb{R} \right \} \cup \left \{\mathbbm{1}_{\{\ell < a\}}| \ell \in \mathcal{L}(\mathbb{R}^d,\mathbb{R}), a \in \mathbb{R} \right \} } \right \},
  \]
  where $\mathcal{L}(\mathbb{R}^d,\mathbb{R})$ denotes the set of linear maps from $\mathbb{R}^d$ to $\mathbb{R}$. We may write
   \begin{align*}
   \mathcal{G}'(r) \subset \mathcal{H}\times \prod_{i=1}^{2(k-1)}\mathcal{P}_d 
   \times \mathbbm{1}_{\mathcal{B}(0,M)},
   \end{align*}   
   It is well known that
   \[
   d_p\left (\mathcal{P}_d \right) = d.
   \]
    Since every set of functions in this decomposition is composed of functions taking values in $[-1,1]$, we intend to apply Theorem \ref{Mendelson} to every set. Consequently it remains to bound from above the pseudo-dimensions of these sets of functions.
 
 First we deal with $\mathcal{H}$:
 \begin{lem}\label{H} One has
 \[
 d_p(\mathcal{H}) = d_p \left (\left \{ f \mathbbm{1}_{\{|f| \leq F_r\}}| f \in \mathcal{A}ff(\mathbb{R}^d,\mathbb{R}) \right \}\right ) \leq 24(3d+4).
 \]
 \end{lem}
 
 \begin{proof}[Proof of Lemma \ref{H}]
 The first equality is obvious, so we only have to deal with the inequality. We recall that the pseudo-dimension of the set of functions $\left \{ f \mathbbm{1}_{\{|f| \leq F_r\}}| f \in \mathcal{A}ff(\mathbb{R}^d,\mathbb{R}) \right \}$ is the Vapnik dimension of the set of functions 
 \[
 \mathcal{H}'=\left \{ \mathbbm{1} _{\{f \mathbbm{1}_{\{ |f| \leq F_r\}}-t \leq 0\}}| f \in \mathcal{A}ff(\mathbb{R}^d,\mathbb{R}), t \in \mathbb{R} \right \}.
 \]
 
 Let $x_1, \hdots, x_{2m}$ denote $2m$ points in $\mathbb{R}^d$. Since $F_r(x) = c_1 + c_2 \mathbbm{1}_{N^*(x)}$, where $c_1$ et $c_2$ are constants, at least $m$ points fall in an area on which $F_r$ takes the form $F_r(x) = c$, for some constant $c$. Without loss of generality, we assume that $x_1, \hdots, x_m$ fall in such an area. Consequently, we have to bound from above the quantity $\left |\left \{\mathbbm{1} _{\{f \mathbbm{1}_{\{ |f| \leq c\}}  -t \leq 0\}} \right \}(x_1, \hdots, x_m) \right |$.
 Observing that
 \[
 \left \{ \mathbbm{1}_{\{|f| \leq c \}} \right \} = \left \{ \mathbbm{1}_{\{f \leq c\}} \times \mathbbm{1}_{\{f \geq -c\}}  \right \},
 \]
 we deduce that
 \begin{multline*}
 \left |\left \{\mathbbm{1} _{\{f \mathbbm{1}_{\{ |f| \leq c \}}  -t \leq 0\}} \right \}(x_1, \hdots, x_m) \right | \\
 \leq \left |\left \{\mathbbm{1} _{\{f \mathbbm{1}_{ \{f \leq c\}}  -t \leq 0\}} \right \}(x_1, \hdots, x_m) \right | \times \left |\left \{\mathbbm{1} _{\{f \mathbbm{1}_{\{ f \geq -c\}}  -t \leq 0\}} \right \}(x_1, \hdots, x_m) \right |.
 \end{multline*}
 
 Noticing that $d_{VC} \left ( \left \{ \mathbbm{1}_{\{f \leq c\}}| f \in \mathcal{A}ff(\mathbb{R}^d,\mathbb{R}) \right \} \right ) = d+1$ and making use of Sauer's lemma leads to, provided that $m \geq d+1$,
   \[
   \left |\left \{ \mathbbm{1}_{ \{f \leq c\}}  \right \}(x_1, \hdots, x_m) \right | = \left |\left \{ \mathbbm{1}_{ \{ f \geq -c}  \right \}(x_1, \hdots, x_m) \right | \leq \left ( \frac{em}{d+1} \right )^{(d+1)},
   \]
   which ensures that 
   \[
   \left |\left \{\mathbbm{1}_{\{ |f| \leq c\}}  \right \}(x_1, \hdots, x_m) \right | \leq \left ( \frac{em}{d+1} \right )^{2(d+1)}.
   \]
 
 Choose a configuration of  $\left \{ \mathbbm{1}_{|f(x_1)| \leq c}, \hdots, \mathbbm{1}_{|f(x_m)| \leq c} \right \}$ , for instance by indexing the $x_i$'s so that  $|f(x_1)| >c, \hdots, |f(x_r)| >c$ and  $|f(x_{r+1})| \leq c, \hdots, |f(x_m)| \leq c$. For the $r$ first $x_i$'s, only two configurations remains for $\mathcal{H}'(x_1, \hdots, x_r)$, the configuration $(0,\hdots, 0)$ and $(1, \hdots, 1)$.  Considering the $m-r+1$ last $x_i$'s, $\left |\mathcal{H}'(x_{m-r+1}, \hdots, x_m) \right | \leq \left |\left \{ \mathbbm{1}_{\{f - t \leq 0\}} \right \}(x_{m-r+1}, \hdots,m) \right |$. Next, $\left |\left \{ \mathbbm{1}_{\{f - t \leq 0\}} \right \}(x_{m-r+1}, \hdots,m) \right | \leq \left |\left \{ \mathbbm{1}_{\{f - t \leq 0\}} \right \}(x_{1}, \hdots,m) \right |$. Eventually, Sauer's lemma guarantees that $\left |\left \{ \mathbbm{1}_{\{f - t \leq 0\}} \right \}(x_{1}, \hdots,m) \right | \leq \left ( \frac{em}{d+2} \right )^{(d+2)}$, provided that $m\geq d+2$.  Consequently, we get, for $m \geq d+2$,
 \[
 \begin{aligned}
 \left | \mathcal{H}'(x_1, \hdots, x_{2m}) \right | & = \left |\left \{ \mathbbm{1} _{\{f \mathbbm{1}_{\{ |f| \leq F_r\}}-t \leq 0\}} \right \}(x_1, \hdots, x_{2m}) \right | \\
                                                    & \leq 2^m \times \left |\left \{\mathbbm{1} _{\{f\mathbbm{1}_{\{ |f| \leq c\}}  -t \leq 0\}} \right \}(x_1, \hdots, x_m) \right | \\
                                                    & \leq 2^m \times \left | \left \{ \mathbbm{1}_{ \{|f|  \leq c\}} \right \}(x_1, \hdots, x_m) \right | \times 2 \left (\frac{em}{d+2} \right)^{d+2} \\
 & \leq 2^m \times 2 \left ( \frac{em}{d+1} \right )^{2(d+1)} \left (\frac{em}{d+2} \right)^{d+2}.
 \end{aligned}
 \] 
     To give an upper bound on $d_p(\mathcal{H})$, we have to find $m \geq d+2$ such that
 \[
 2 \left ( \frac{em}{d+1} \right )^{2(d+1)} \left (\frac{em}{d+2} \right)^{d+2} < 2^m.
 \]
Noticing that $x \mapsto \log_2(x)$ is a strictly concave map, we deduce that 
 \[
 2(d+1)\log_2\left ( \frac{em}{d+1}\right ) + (d+2) \log_2 \left ( \frac{em}{d+2} \right ) < (3d+4)\log_2\left(\frac{3em}{3d+4}\right).
 \]
 Consequently, a sufficient condition on $m$ is given by
 \[
 \left(\frac{3em}{3d+4}\right )^{3d+4} \leq 2^{m-1}.
 \]
 Using the same method as in \cite{Baum89}, the choice $m=  \left\lceil 3(3d+4)\log_2(3e) \right\rceil$ turns out to be adequate. At last, noticing that $2m \leq 24(3d+4)$, we immediately deduce that $d_p(\mathcal{H}) \leq 2m \leq 24(3d+4)$.
 \end{proof}
 Applying Lemma \ref{multiplication} and Theorem \ref{Mendelson} yields that
 \[
 \begin{aligned}
 \mathcal{N}\left ( \mathcal{G}'(r), \varepsilon \right ) & \leq \mathcal{N}\left (\mathcal{H},\varepsilon/(2k-1) \right ) \times {\mathcal{N}\left ( \mathcal{P}_d,\varepsilon/(2k-1) \right )}^{2k-2} \\
 & \leq \left ( \frac{4k-2}{\varepsilon} \right)^{K[24(3d+4) + (d+1)(2k-2)]}.
 \end{aligned}
 \] 
 At last, the result of Proposition \ref{chaining} is given by
 \[
 \begin{aligned}
 \mathcal{N}\left ( \mathcal{G}(r), \varepsilon \right ) & \leq \mathcal{N}\left ( \mathcal{G}'(r), \varepsilon/k^2 \right )^{k^2} \\
 & \leq  \left ( \frac{2(2k-1)k^2}{\varepsilon} \right )^{Kk^2[24(3d+4) + 2(d+1)(k-1)]}.
 \end{aligned}
 \]

\subsection{Proof of Proposition \ref{definitionDelta}}\label{ProofofPropositiondefinitionDelta} 

The proof of Proposition \ref{definitionDelta} is based on elementary properties of distributions with finite support, which are extended to the case where the source distribution is supported on small balls. Throughout this subsection, a source distribution $P_{\sigma'}$ is fixed, so that $R(Q,P_{\sigma'})$ may be denoted by $R(Q)$.

\begin{lem}\label{cellulesàdeuxpoints}
Let $z_1$ and $z_2$ be points in $\mathbb{R}^d$, denote by $R$ the quantity $\| z_1 - z_2 \|$, by $U_i$ the ball $\mathcal{B}(z_i,\rho)$. At last, let $P$ denote the cone-shaped distribution with density 
\[
\frac{2(d+1)}{V} (\mathbb{1}_{\|x - z_i \| \leq \rho} (\rho - \|x-z_i\|),
\]
over each ball $U_i$, where $V$ denote the volume of the unit ball. Then, if
\[
(R/2 - 3 \rho)^2 \geq \rho^2 \frac{2  d(d+1)}{(d+2)(d+3)} \quad \mbox{and} \quad \rho \leq \frac{R}{2},
\]
then the best $2$-quantizer $Q^*_2$ is such that $Q^*_2(U_i) = z_i$ for $i=1,2$. Furthermore, the best $1$-quantizer $Q_1^*$ is such that $Q_1^*(U_1 \cup U_2) = (z_1 + z_2)/2$. 
\end{lem}

\begin{proof}[Proof of Lemma \ref{cellulesàdeuxpoints}]
   Let $V_i$ denote the Voronoi cell associated with $z_i$ in the Voronoi diagram generated by $(z_1,z_2)$. Denote by $Q^*_2$ the quantizer satisfying $Q^*_2(U_i) = z_i$ for $i=1,2$. 
   
   For any quantizer $Q$ denote by $R_i(Q) = \int_{V_i}{\|x - Q(x)\|^2dx}$ the contribution of the cell $i$ to the distortion of $Q$. Denote by $V$ the volume of the unit ball, and by $S$ its surface. Recalling that $S = d \times V$, an elementary calculation shows that
   \[
   \begin{aligned}
   R_i(Q^*_2) & = \frac{1}{2} \frac{d+1}{\rho^{d+1} V} \int_{0}^{\rho}{S (\rho r^{d+1} - r^{d+2}) dr} \\
            & = \rho^2 \frac{d(d+1)}{2(d+2)(d+3)}.
            \end{aligned}
            \]
            Let $m_i^{in} = |Q(U_i) \cap V_i|$ and $m_i^{out} = |Q(U_i)) \cap V_i^c|$ denote the number of images of $U_i$ sent inside and outside $V_i$. For a given $i$, there are three situations of interest, which are described below.
            \begin{enumerate}
            \item $m_i^{out} = 0$ and $m_i^{in} = 1$, then it is clear that $R_i(Q_2^*) \leq R_i(Q)$.
            \item $m_i^{out} = 0$ and $m_i^{in} = 2$, then $R_i(Q) \geq 0 = R_i(Q^*) - \rho^2 \frac{d(d+1)}{2(d+2)(d+3)}$.
            \item $m_i^{out} \geq 1$, then there exists $z$ $\in$ $U_i$ such that $Q(z) \notin V_i$. Consequently, $\|z - Q(z)\| \geq \frac{R}{2} - \rho$.  Let $z'$ $\in$ $\mathcal{B}(z_i,\rho)$, then 
            \[
            \|z' - Q(z')\| \geq \| z - Q(z') \| - 2 \rho \geq \|z - Q(z) \| - 2 \rho \geq \frac{R}{2} - 3\rho.
            \]
             Hence we deduce 
             \[
             R_i(Q) \geq 1/2 (\frac{R}{2} - 3\rho)^2 = R_i(Q^*_2) + 1/2 \left( (\frac{R}{2} - 3\rho)^2 - \rho^2 \frac{d(d+1)}{(d+2)(d+3)} \right ) .
            \]
            \end{enumerate}
            
                  Since $Q$ is a $2$-quantizer, it is easy to see that 
                  \[
                  |\left\{i;m_i^{in} \geq 2\right\}| \leq |\left\{i;m_i^{out} \geq 1\right\}|.
                   \]
                   From this we deduce that
                  \[
                  \begin{aligned}
                  R(Q) & = \sum_{\{i;m_i^{in} \geq 2,m_i^{out} =0 \}}{R_i(Q)} + \sum_{\{i;m_i^{out} \geq 1\}}{R_i(Q)} + \sum_{\{i;m_i^{in} =1,m_i^{out} =0\}}{R_i(Q)} \\
                                                        & \geq R(Q^*_2) + \sum_{\{i;m_i^{in} \geq 2,m_i^{out} =0\}}{\frac{1}{2} \left ( \left ( \frac{R}{2} - 3 \rho \right )^2 - \rho^2 \frac{d(d+1)}{(d+2)(d+3)} \right )}.
                                                        \end{aligned}
                                                        \]
        Taking $(R/2 - 3 \rho)^2 \geq \rho^2 \frac{2  d(d+1)}{(d+2)(d+3)}$ ensures that $R(Q) \geq R(Q_2^*)$.                                                                  \end{proof}
        
Considering the distributions $P_\sigma$, $\sigma$ in $\{-1,+1\}^{m}$, taking $\rho \leq \frac{\Delta}{16}$ ensures that the conditions of Lemma \ref{cellulesàdeuxpoints} are satisfied when considering $P_{\sigma| U_i \cup U'_i}$. We turn now to the proof of Proposition \ref{definitionDelta}.

Let $Q$ be a $k$-quantizer. The following construction provides $Q_\sigma$ $\in$ $\mathcal{Q}$ such that $R(Q_\sigma) \leq R(Q)$. Let $V_i$ denote the union of the Voronoi cells associated with $z_i$ and $z_i + \omega_i$, in the Voronoi diagram generated by the sequences $z$ and $\omega$. We adopt the following notation
\[
\left \{
\mbox{
\begin{tabular}{@{}ccl}
$n_i(Q)$ & $=$ & $\left |Q(\mathcal{B}(0,M)) \cap V_i\right |$\\
$n_i^{out}(Q)$ & $=$& $ \left |Q(V_i) \cap V_i^c \right |$ \\
$I_j(Q)$ & $=$ & $\left \{ i; n_i(Q) = j \right \}$\\
$i_j(Q)$ & $=$ & $\left | I_j(Q) \right |$ \\
$i_{\geq j}(Q)$ & $=$ & $\underset{i \geq j}{\sum}{i_j (Q)}.$
\end{tabular}
}
\right .
\]

The first step is to add code points to empty cells. From the $k$-quantizer $Q$, a quantizer $Q_1$ is built as follows
\begin{itemize}
\item If $n_i(Q) \geq 1$, then we take $Q_{1|V_i} \equiv Q_{|V_i}$.
\item If $n_i(Q) =0$, then we set $Q_1(U_i) = Q_1(U'_i) = z_i + \frac{w_i}{2}$.
\end{itemize}
Notice that $Q_1$ is a $(k + i_0(Q))$-quantizer. Denote $k_1 = k + i_0$ and $p_{\pm} = \frac{1 \pm \delta}{2m}$, then $R(Q_1)$ can be bounded as follows.

 Let $i$ be an integer between $1$ and $m$. We denote by $R_i(Q)$ the contribution of $V_i$ to the risk $R(Q)$. If $i \in I_{\geq 1}$, then $R_i(Q) = R_i(Q_1)$. Otherwise, if $i \in I_0(Q)$, 
\[
R_i(Q_1) = 2 p_{\pm} \rho^2 \frac{d(d+1)}{(d+2)(d+3)} + p_{\pm} \frac{\Delta^2}{2}.
\]
Furthermore, if $i \in I_0$, then $n_i^{out}(Q) \geq 1$, which ensures that, as in the proof of Lemma \ref{cellulesàdeuxpoints}, 
\[
R_i(Q) \geq p_{\pm}\left ( \frac{(A-2)\Delta}{2} - 2 \rho \right )^2.
\]
Since $A \geq 6 $ and $\rho \leq \frac{\Delta}{16}$, we may write 
\[
\begin{aligned}
R_i(Q) - R_i(Q_1) & \geq p_{\pm} \left [(2\Delta - 2\rho)^2 - 4 \rho^2 - \frac{\Delta^2}{2}\right] \\
                  & \geq p_{\pm} \left [ 2 \Delta( \frac{3\Delta}{4}) - \frac{\Delta^2}{2} \right ] \\
                  & \geq p_{-} \frac{3\Delta^2}{2}.
                  \end{aligned}
                  \]
                  Summing all the contributions of $V_i$'s leads to
                  \[
R(Q_1) \leq R(Q) - i_0(Q) p_{-} \frac{3 \Delta ^2}{2}
\]
     Next, we build the quantizer $Q_2$ according to the following rule:
                       \begin{itemize}
                  \item If $n_i(Q_1) \geq 2$, then $Q_2(U_i) = z_i$ and $Q_2(U'_i) = z_i + w_i$.
                  \item If $n_i(Q_1) = 1$, then $Q_2(U_i) = Q_2(U'_i) = z_i + \frac{w_i}{2}$.
                  \end{itemize}
                  Since for $i=1, \hdots, k$, $n_i(Q_1) \geq 1$, $Q_2$ is defined on every $V_i$. Notice that, since $I_j(Q_1) = I_j(Q)$ for $j \geq 2$, $Q_2$ has $k_2 = k+i_0(Q) - \sum_{p\geq3}{(p-2)i_p(Q)}$ code points. The following lemma offers a relation between $R(Q_2)$ and $R(Q_1)$.
                  \begin{lem}\label{recentrage}
									One has
                  \[
                  R(Q_2) \leq R(Q_1) + i_{\geq3}(Q) \frac{p_{+} \Delta ^2 }{128}.
                  \]
                  \end{lem}
                  \begin{proof}[Proof of Lemma \ref{recentrage}]
                  Let $i$ be an integer between $1$ and $m$. Several cases may occur, as described below.
                  \begin{itemize}
                  \item Assume that $n_i(Q_1) =1$. 
                  \begin{itemize}
                  \item If $n_i^{out}(Q_1) = 0$, then $R_i(Q_1) \geq R_i(Q_2)$, according to Lemma \ref{cellulesàdeuxpoints}.
                  \item If $n_i^{out}(Q_1) \geq 1$, then, using the same technique as mentioned to bound $R(Q_1)$ from above, $R_i(Q_1) - R_i(Q_2) \geq p_{\pm}\frac{3 \Delta^2}{2}$, which leads to $R_i(Q_1) \geq R_i(Q_2)$.
                  \end{itemize}
                  \item Assume that $n_i(Q_1) = 2$.
                  \begin{itemize}
                  \item If $n_i^{out}(Q_1) = 0$, then $R_i(Q_1) \geq R_i(Q_2)$, according to Lemma \ref{cellulesàdeuxpoints}.
                  \item If $n_i^{out}(Q_1) \geq 1$, then, since $R_i(Q_2) = 2 p_{\pm} \frac{\rho^2 d}{d+2} \leq p_{+} \frac{\Delta^2}{128}$, $R_i(Q_1) - R_i(Q_2) \geq \Delta ^2 \geq 0$.
                  \end{itemize}
                  \item At last, assume that $n_i(Q_i) \geq 3$. If $n_i^{out}(Q_1) \geq 1$, then $R_i(Q_1) \geq R_i(Q_2)$. If $n_i^{out}(Q_1) =0$, then $R_i(Q_1) \geq 0 = R_i(Q_1) - 2 p_{\pm} \frac{\Delta^2}{128}$. In both cases $R(Q_2) \leq R(Q_1) + p_{+} \frac{\Delta^2}{128}$.
                  \end{itemize}
                  Noticing that $I_{\geq3}(Q_1) = I_{\geq 3}(Q)$, and summing the contributions $R_i(Q_2)$ leads to the desired result.
                  \end{proof}
                  
                  The last step is to build a quantizer $Q_\sigma$ from $Q_2$ with exactly $k$ code points.
                  \begin{itemize}
                  \item If $k_2$ = $k$, set $Q_\sigma = Q_2$.
                  \item If $k_2 < k$, choose $(k-k_2)$ $V_i$ such that $n_i(Q_2) = 1$ (elementary calculation shows that there exist at least $k - k_2$ such $V_i$'s). For every such $V_i$, set $Q_\sigma(U_i) = z_i$ and $Q_\sigma(U'_i) = z_i + \omega_i$. Then
                  \[
                  R(Q_\sigma) \leq R(Q_2) - (k-k_2)p_{-} \frac{\Delta^2}{2}.
                  \]
                  \item If $k_2 > k$, choose $(k_2 - k)$ cells $V_i$ such that $n_i(Q_2) = 2$. For every such $V_i$, define $Q_\sigma(U_i) = Q_\sigma(U'_i) = z_i + \frac{\omega_i}{2}$. Then
                  \[
                  R(Q_\sigma) \leq R(Q_2) + (k_2-k)p_+ \frac{\Delta^2}{2}.
                  \]
                  \end{itemize}
                 By construction, $Q_\sigma$ has exactly $k$ code points, and is an element of $\mathcal{Q}$. Finally, a result on the risk of $Q_\sigma$ is given by the following proposition.
                  \begin{prop}\label{comparaisonquantificateurs}
                  Let $Q$ be a quantizer and $Q_\sigma$ built as mentioned above. Then, 
                  \[                 
                  R(Q_\sigma) \leq R(Q).
                  \]
                  \end{prop}
									\begin{proof}[Proof of Proposition \ref{comparaisonquantificateurs}]
                  Since $\delta \leq \frac{1}{3}$, easy calculation ensures that $1- \frac{p_-}{p_+} \leq \frac{1}{2}$. 
                  
                  Suppose that $k_2 \leq k$. Comparing the risk of $Q$ to the risks of $Q_1$, $Q_2$ and $Q_\sigma$ leads to
                  \[
                  R(Q_\sigma)  \leq R(Q) - i_0 p_- \frac{3 \Delta^2}{2} + (i_0 + 2 i_{\geq 3} - \sum_{p \geq 3}{p i_p}) p_- \frac{\Delta^2}{2} + i_{\geq 3} p_+ \frac{\Delta^2}{128}.
                  \]
                  Since $\sum_{p\geq3}{p i_p} \geq 3 i_{\geq 3}$, it is clear that
                  \[
                  \begin{aligned}
                  R(Q_\sigma)  & \leq R(Q) - p_- i_0 \frac{\Delta^2}{2} + \Delta^2 i_{\geq3} (\frac{p_+}{128} - \frac{p_-}{2}) \\
                                & \leq R(Q).
                  \end{aligned}              
                  \]                 
                  Next, suppose that $k_2 > k$. Then 
                  \[
                  \begin{aligned}
                  R(Q_\sigma) & \leq R(Q) + \left ( i_0 + 2 i_{\geq3} - \sum_{p \geq 3}{p i_p} \right ) p_+ \frac{\Delta^2}{2} + i_{\geq3} p_+ \frac{\Delta^2}{128} - i_0 p_- \frac{3 \Delta ^2}{2}\\
                  & \leq R(Q) + i_0 \frac{\Delta^2}{2} (p_+ - 3 p_-) + p_+ i_{\geq 3} \Delta^2(\frac{1}{128} - \frac{1}{2}),
                  \end{aligned}
                  \] 
                  which yields that $R(Q_\sigma) \leq R(Q)$.
                  \end{proof}                    

\subsection{Proof of Lemma \ref{Assouadtechnique}}\label{Proof of Lemma Assouadtechnique}

Let introduce, for distributions $P$ and $Q$ with densities $f$ and $g$ the affinity 
\[
\alpha(P,Q) = \int{\sqrt{fg}},
\]
so that $H^2(P,Q)=2(1-\alpha(P,Q))$. Elementary calculation shows that, if $\rho(\sigma, \sigma')=4$, then
\[
\alpha(P_\sigma, P_{\sigma'})= 1 + \frac{2}{m}\left (\sqrt{1-\delta^2} -1 \right ) \geq 1- \frac{2\delta^2}{m}.
\]
Hence we deduce
\begin{align*}
H^2(P_\sigma^{\otimes n},P_{\sigma'}^{\otimes n}) & = 2(1-\alpha(P_\sigma^{\otimes n},P_{\sigma'}^{\otimes n})) \\
                                                  &= 2(1-\alpha^n(P_\sigma,P_{\sigma'})) \\
																									& \leq \frac{4n\delta^2}{m}.
\end{align*}
Finally, since $\rho(\tau,\tau')=2$ implies $\rho(\sigma(\tau),\sigma(\tau'))=4$, for $\tau$, $\tau'$ in $\{-1,+1\}^{\frac{m}{2}}$,  the first part of Lemma \ref{Assouadtechnique} is proved.

Next, for simplicity assume that $\sigma$ is such that $\sigma_1 = \hdots \sigma_{\frac{m}{2}} = +1$ and $\sigma_{\frac{m}{2}+1} = \hdots = \sigma_m =-1$. Let $\mathcal{S}^-$ and $\mathcal{S}^+$denote the set of mistakes of $\sigma'$, that is 
\begin{align*}
\left \{
\begin{array}{@{}ccl}
\mathcal{S}^-&=&\{i=1, \hdots, \frac{m}{2}|\quad \sigma'_i=-1\} \\
\mathcal{S}^+&=&\{i=\frac{m}{2}+1, \hdots, \m|\quad \sigma'_i=+1\}. 
\end{array}
\right.
\end{align*}
      Finally let $s^+$ and $s^-$ respectively denote $|\mathcal{S}^+|$ and $|\mathcal{S}^-|$. Since $\sum_{i=1}^{m}{\sigma'_i}=0$, it is clear that $s^+ = s^- := s$.
			
			As in Subsection \ref{ProofofPropositiondefinitionDelta}, let $R_i(Q_{\sigma'})$ denote the contribution of $U_i \cup U'_i$ to the distortion.
Then, for $i$ in $\mathcal{S}^-$, elementary calculation shows that
\[
R_i(Q_{\sigma'}) = R_i(Q_\sigma) + \frac{(1+\delta)\Delta^2}{4m}.
\]
Symmetrically, for $i$ in $\mathcal{S}^+$,
\[
R_i(Q_{\sigma'}) = R_i(Q_\sigma) - \frac{(1-\delta)\Delta^2}{4m}.
\]
Summing with respect to $i$ and taking into account that $s^+ = s^- =s$ leads to
\begin{align*}
R(Q_{\sigma'}) = R(Q_\sigma) + s \frac{\Delta^2 \delta}{2m}.
\end{align*}
Remarking that $s= \frac{\rho(\sigma,\sigma')}{4}$ concludes the proof of Lemma \ref{Assouadtechnique}.

         \bibliography{biblio}

\begin{thebibliography}{21}

\bibitem{Antos05}
\begin{barticle}[author]
\bauthor{\bsnm{Antos},~\bfnm{Andr{\'a}s}\binits{A.}}
(\byear{2005}).
\btitle{Improved minimax bounds on the test and training distortion of
  empirically designed vector quantizers}.
\bjournal{IEEE Trans. Inform. Theory}
\bvolume{51}
\bpages{4022--4032}.
\bdoi{10.1109/TIT.2005.856980}
\bmrnumber{2239018 (2007b:94149)}
\end{barticle}
\endbibitem

\bibitem{Antos04}
\begin{barticle}[author]
\bauthor{\bsnm{Antos},~\bfnm{Andr{\'a}s}\binits{A.}},
  \bauthor{\bsnm{Gy{\"o}rfi},~\bfnm{L{\'a}szl{\'o}}\binits{L.}} \AND
  \bauthor{\bsnm{Gy{\"o}rgy},~\bfnm{Andr{\'a}s}\binits{A.}}
(\byear{2005}).
\btitle{Individual convergence rates in empirical vector quantizer design}.
\bjournal{IEEE Trans. Inform. Theory}
\bvolume{51}
\bpages{4013--4022}.
\bdoi{10.1109/TIT.2005.856976}
\bmrnumber{2239017 (2007a:94125)}
\end{barticle}
\endbibitem

\bibitem{Bartlett98}
\begin{barticle}[author]
\bauthor{\bsnm{Bartlett},~\bfnm{Peter~L.}\binits{P.~L.}},
  \bauthor{\bsnm{Linder},~\bfnm{Tam{\'a}s}\binits{T.}} \AND
  \bauthor{\bsnm{Lugosi},~\bfnm{G{\'a}bor}\binits{G.}}
(\byear{1998}).
\btitle{The minimax distortion redundancy in empirical quantizer design}.
\bjournal{IEEE Trans. Inform. Theory}
\bvolume{44}
\bpages{1802--1813}.
\bdoi{10.1109/18.705560}
\bmrnumber{1664098 (2001f:94006)}
\end{barticle}
\endbibitem

\bibitem{Baum89}
\begin{barticle}[author]
\bauthor{\bsnm{Baum},~\bfnm{Eric~B.}\binits{E.~B.}} \AND
  \bauthor{\bsnm{Haussler},~\bfnm{David}\binits{D.}}
(\byear{1989}).
\btitle{What size net gives valid generalization?}
\bjournal{Neural Comput.}
\bvolume{1}
\bpages{151--160}.
\bdoi{10.1162/neco.1989.1.1.151}
\end{barticle}
\endbibitem

\bibitem{Biau08}
\begin{barticle}[author]
\bauthor{\bsnm{Biau},~\bfnm{G{\'e}rard}\binits{G.}},
  \bauthor{\bsnm{Devroye},~\bfnm{Luc}\binits{L.}} \AND
  \bauthor{\bsnm{Lugosi},~\bfnm{G{\'a}bor}\binits{G.}}
(\byear{2008}).
\btitle{On the performance of clustering in {H}ilbert spaces}.
\bjournal{IEEE Trans. Inform. Theory}
\bvolume{54}
\bpages{781--790}.
\bdoi{10.1109/TIT.2007.913516}
\bmrnumber{2444554 (2009m:68221)}
\end{barticle}
\endbibitem

\bibitem{Blanchard08}
\begin{barticle}[author]
\bauthor{\bsnm{Blanchard},~\bfnm{Gilles}\binits{G.}},
  \bauthor{\bsnm{Bousquet},~\bfnm{Olivier}\binits{O.}} \AND
  \bauthor{\bsnm{Massart},~\bfnm{Pascal}\binits{P.}}
(\byear{2008}).
\btitle{Statistical performance of support vector machines}.
\bjournal{Ann. Statist.}
\bvolume{36}
\bpages{489--531}.
\bdoi{10.1214/009053607000000839}
\bmrnumber{2396805 (2009m:62085)}
\end{barticle}
\endbibitem

\bibitem{Chichi13}
\begin{bunpublished}[author]
\bauthor{\bsnm{Chichignoud},~\bfnm{Micha{\"e}l}\binits{M.}} \AND
  \bauthor{\bsnm{Loustau},~\bfnm{S{\'e}bastien}\binits{S.}}
(\byear{2013-06}).
\btitle{{Adaptive Noisy Clustering}}.
\end{bunpublished}
\endbibitem

\bibitem{Fischer10}
\begin{barticle}[author]
\bauthor{\bsnm{Fischer},~\bfnm{Aur{\'e}lie}\binits{A.}}
(\byear{2010}).
\btitle{Quantization and clustering with {B}regman divergences}.
\bjournal{J. Multivariate Anal.}
\bvolume{101}
\bpages{2207--2221}.
\bdoi{10.1016/j.jmva.2010.05.008}
\bmrnumber{2671211 (2012c:62188)}
\end{barticle}
\endbibitem

\bibitem{Gersho91}
\begin{bbook}[author]
\bauthor{\bsnm{Gersho},~\bfnm{Allen}\binits{A.}} \AND
  \bauthor{\bsnm{Gray},~\bfnm{Robert~M.}\binits{R.~M.}}
(\byear{1991}).
\btitle{Vector quantization and signal compression}.
\bpublisher{Kluwer Academic Publishers}, \baddress{Norwell, MA, USA}.
\end{bbook}
\endbibitem

\bibitem{GL00}
\begin{bbook}[author]
\bauthor{\bsnm{Graf},~\bfnm{Siegfried}\binits{S.}} \AND
  \bauthor{\bsnm{Luschgy},~\bfnm{Harald}\binits{H.}}
(\byear{2000}).
\btitle{Foundations of quantization for probability distributions}.
\bseries{Lecture Notes in Mathematics}
\bvolume{1730}.
\bpublisher{Springer-Verlag}, \baddress{Berlin}.
\bdoi{10.1007/BFb0103945}
\bmrnumber{1764176 (2001m:60043)}
\end{bbook}
\endbibitem

\bibitem{Koltchinskii04}
\begin{barticle}[author]
\bauthor{\bsnm{Koltchinskii},~\bfnm{Vladimir}\binits{V.}}
(\byear{2006}).
\btitle{Local {R}ademacher complexities and oracle inequalities in risk
  minimization}.
\bjournal{Ann. Statist.}
\bvolume{34}
\bpages{2593--2656}.
\bdoi{10.1214/009053606000001019}
\bmrnumber{2329442 (2009h:62060)}
\end{barticle}
\endbibitem

\bibitem{Levrard12}
\begin{barticle}[author]
\bauthor{\bsnm{Levrard},~\bfnm{Cl{\'e}ment}\binits{C.}}
(\byear{2013}).
\btitle{Fast rates for empirical vector quantization}.
\bjournal{Electron. J. Stat.}
\bvolume{7}
\bpages{1716--1746}.
\bdoi{10.1214/13-EJS822}
\end{barticle}
\endbibitem

\bibitem{Linder02}
\begin{bincollection}[author]
\bauthor{\bsnm{Linder},~\bfnm{Tam{\'a}s}\binits{T.}}
(\byear{2002}).
\btitle{Learning-theoretic methods in vector quantization}.
In \bbooktitle{Principles of nonparametric learning ({U}dine, 2001)}.
\bseries{CISM Courses and Lectures}
\bvolume{434}
\bpages{163--210}.
\bpublisher{Springer}, \baddress{Vienna}.
\bmrnumber{1987659 (2004f:68128)}
\end{bincollection}
\endbibitem

\bibitem{Linder94}
\begin{barticle}[author]
\bauthor{\bsnm{Linder},~\bfnm{Tam{\'a}s}\binits{T.}},
  \bauthor{\bsnm{Lugosi},~\bfnm{G{\'a}bor}\binits{G.}} \AND
  \bauthor{\bsnm{Zeger},~\bfnm{Kenneth}\binits{K.}}
(\byear{1994}).
\btitle{Rates of convergence in the source coding theorem, in empirical
  quantizer design, and in universal lossy source coding}.
\bjournal{IEEE Trans. Inform. Theory}
\bvolume{40}
\bpages{1728--1740}.
\bdoi{10.1109/18.340451}
\bmrnumber{1322387 (96b:94005)}
\end{barticle}
\endbibitem

\bibitem{Tsybakov99}
\begin{barticle}[author]
\bauthor{\bsnm{Mammen},~\bfnm{Enno}\binits{E.}} \AND
  \bauthor{\bsnm{Tsybakov},~\bfnm{Alexandre~B.}\binits{A.~B.}}
(\byear{1999}).
\btitle{Smooth discrimination analysis}.
\bjournal{Ann. Statist.}
\bvolume{27}
\bpages{1808--1829}.
\bdoi{10.1214/aos/1017939240}
\bmrnumber{1765618 (2001i:62074)}
\end{barticle}
\endbibitem

\bibitem{Massart03}
\begin{bbook}[author]
\bauthor{\bsnm{Massart},~\bfnm{Pascal}\binits{P.}}
(\byear{2007}).
\btitle{Concentration inequalities and model selection}.
\bseries{Lecture Notes in Mathematics}
\bvolume{1896}.
\bpublisher{Springer}, \baddress{Berlin}.
\bnote{Lectures from the 33rd Summer School on Probability Theory held in
  Saint-Flour, July 6--23, 2003, With a foreword by Jean Picard}.
\bmrnumber{2319879 (2010a:62008)}
\end{bbook}
\endbibitem

\bibitem{Massart06}
\begin{barticle}[author]
\bauthor{\bsnm{Massart},~\bfnm{Pascal}\binits{P.}} \AND
  \bauthor{\bsnm{N{\'e}d{\'e}lec},~\bfnm{{\'E}lodie}\binits{{\'E}.}}
(\byear{2006}).
\btitle{Risk bounds for statistical learning}.
\bjournal{Ann. Statist.}
\bvolume{34}
\bpages{2326--2366}.
\bdoi{10.1214/009053606000000786}
\bmrnumber{2291502 (2009e:62282)}
\end{barticle}
\endbibitem

\bibitem{Mendelson03}
\begin{barticle}[author]
\bauthor{\bsnm{Mendelson},~\bfnm{S.}\binits{S.}} \AND
  \bauthor{\bsnm{Vershynin},~\bfnm{R.}\binits{R.}}
(\byear{2003}).
\btitle{Entropy and the combinatorial dimension}.
\bjournal{Invent. Math.}
\bvolume{152}
\bpages{37--55}.
\bdoi{10.1007/s00222-002-0266-3}
\bmrnumber{1965359 (2004d:60047)}
\end{barticle}
\endbibitem

\bibitem{Pollard82b}
\begin{barticle}[author]
\bauthor{\bsnm{Pollard},~\bfnm{David}\binits{D.}}
(\byear{1982}).
\btitle{Quantization and the method of k -means}.
\bjournal{IEEE Transactions on Information Theory}
\bvolume{28}
\bpages{199-204}.
\end{barticle}
\endbibitem

\bibitem{Pollard82}
\begin{barticle}[author]
\bauthor{\bsnm{Pollard},~\bfnm{David}\binits{D.}}
(\byear{1982}).
\btitle{A central limit theorem for {$k$}-means clustering}.
\bjournal{Ann. Probab.}
\bvolume{10}
\bpages{919--926}.
\bmrnumber{672292 (84c:60047)}
\end{barticle}
\endbibitem

\bibitem{Tsybakov09}
\begin{bbook}[author]
\bauthor{\bsnm{Tsybakov},~\bfnm{Alexandre~B.}\binits{A.~B.}}
(\byear{2009}).
\btitle{Introduction to nonparametric estimation}.
\bseries{Springer Series in Statistics}.
\bpublisher{Springer}, \baddress{New York}.
\bnote{Revised and extended from the 2004 French original, Translated by
  Vladimir Zaiats}.
\bdoi{10.1007/b13794}
\bmrnumber{2724359 (2011g:62006)}
\end{bbook}
\endbibitem

\end{thebibliography}
        
\end{document}